\documentclass[final]{siamart190516}

\usepackage[utf8]{inputenc}
\usepackage[T1]{fontenc}
\usepackage[english]{babel}
\usepackage{amsfonts}
\usepackage{amsopn}
\usepackage[shortlabels]{enumitem}
\usepackage{graphicx}
\usepackage{booktabs}
\usepackage{tabularx}
\usepackage{geometry}
\usepackage{makeidx}
\usepackage{fancyhdr}
\usepackage[autostyle,italian=guillemets]{csquotes}
\usepackage{cite}
\usepackage{courier}
\usepackage{alltt}
\usepackage{colonequals}
\usepackage{ntheorem}
\usepackage[framed,numbered,autolinebreaks]{mcode} 
\usepackage{hyperref}
\usepackage{theoremref}
\usepackage{xargs} 
\usepackage[]{xcolor}
\usepackage{algorithm, algorithmic}
\usepackage{multirow}
\usepackage{multicol}
\usepackage[colorinlistoftodos, prependcaption, textsize = small]{todonotes}
\usepackage{siunitx}
\usepackage[normalem]{ulem}

\newcommand{\bb}{\mathbf b}

\newcommand{\ff}{\mathbf f}

\newcommand{\nn}{\mathbf n}

\newcommand{\rr}{\mathbf r}
\renewcommand{\ss}{\mathbf s}
\renewcommand{\tt}{\mathbf t}
\newcommand{\uu}{\mathbf u}
\newcommand{\vv}{\mathbf v}
\newcommand{\ww}{\mathbf w}

\renewcommand{\AA}{\mathbf A}
\newcommand{\BB}{\mathbf B}

\newcommand{\DD}{\mathbf D}

\newcommand{\FF}{\mathbf F}

\newcommand{\II}{\mathbf I}

\newcommand{\MM}{\mathbf M}

\newcommand{\RR}{\mathbf R}
\renewcommand{\SS}{\mathbf S}
\newcommand{\TT}{\mathbf T}
\newcommand{\UU}{\mathbf U}
\newcommand{\VV}{\mathbf V}
\newcommand{\WW}{\mathbf W}

\newcommand{\Tm}{\mathbf{T}_{\mu}}
\newcommand{\Tz}{\mathbf{T}_{z}}

\newcommand{\E}{\mathcal E}

\newcommand{\I}{\mathcal I}
\newcommand{\J}{\mathcal J}

\renewcommand{\O}{\mathcal O}

\renewcommand{\S}{\mathcal S}

\newcommand{\AAA}{\mathbb A}

\newcommand{\EEE}{\mathbb E}

\newcommand{\III}{\mathbb I}
\newcommand{\JJJ}{\mathbb J}

\newcommand{\NNN}{\mathbb N}

\newcommand{\PPP}{\mathbb P}

\newcommand{\RRR}{\mathbb R}

\newcommand{\AAAu}{\EEE}

\newcommand{\AAAw}{\AAA} 
\newcommand{\PPPw}{\PPP} 
\newcommand{\PPPu}{\JJJ}

\newcommand{\au}{a_E}       
\newcommand{\Au}{\E}        
\newcommand{\Pu}{\J}
\newcommand{\ffu}{\bb}

\newcommand{\ffw}{\ff}

\newcommand{\FFw}{\FF}

\newcommand{\p}{\prime}
\newcommand{\rank}{\mathrm{rank}}
\renewcommand{\d}{\partial}
\newcommand{\mat}{\mathrm{mat}}
\newcommand{\frob}{\mathrm{F}}

\renewcommand{\vec}{\mathrm{vec}}

\newcommand{\duality}[2]{{\langle #1, #2 \rangle}}
\newcommand{\inner}[2]{{\left( #1, #2 \right)}}
\newcommand{\norm}[1]{{\left\Vert #1 \right\Vert}}
\newcommand{\e}{\varepsilon}

\newcommand{\Sd}{\S_{\delta}}

\newcommand{\tr}{\mathrm{tr}}

\newcommand{\bW}{\mathbf{W}}

\newcommand{\bw}{\mathbf{w}}

\renewcommand{\d}{\partial}

\newcommand{\normc}{\norm{\cdot}}

\newsiamremark{remark}{Remark}
\newsiamthm{problem}{Problem}
\newsiamremark{assumption}{Assumption}

\newcommandx{\checked}[2][1=]{\todo[linecolor=green,backgroundcolor=green!25,bordercolor=black,#1]{#2}}

\newcommand{\new}{\textcolor{black}}
\newcommand{\rev}{\textcolor{black}}

\title{Low-rank tensor product Richardson iteration for Radiative Transfer in plane-parallel geometry}

\author{Markus Bachmayr \thanks{Institut f\"ur Geometrie und Praktische Mathematik, RWTH Aachen University, Templergraben 55, 52056 Aachen, Germany. \email{bachmayr@igpm.rwth-aachen.de}}
\and Riccardo Bardin \thanks{Department of Applied Mathematics, University of Twente, P.O. Box 217, 7500 AE Enschede, The Netherlands. \email{r.bardin@utwente.nl}.}
\and Matthias Schlottbom \thanks{Department of Applied Mathematics, University of Twente, P.O. Box 217, 7500 AE Enschede, The Netherlands. \email{m.schlottbom@utwente.nl}}}

\headers{Low-Rank approximations for the RTE}{M. Bachmayr, R. Bardin, M. Schlottbom}

\begin{document}
\maketitle

\begin{abstract}
    The radiative transfer equation (RTE) has been established as a fundamental tool for the description of energy transport, absorption and scattering in many relevant societal applications, and requires numerical approximations. However, classical numerical algorithms scale unfavorably with respect to the dimensionality of such radiative transfer problems, where solutions depend on physical as well as angular variables. 
    In this paper we address this dimensionality issue by developing a low-rank tensor product framework for the RTE in plane-parallel geometry. We exploit the tensor product nature of the phase space to recover an operator equation where the operator is given by a short sum of Kronecker products. This equation is solved by a preconditioned and rank-controlled Richardson iteration in Hilbert spaces. 
    Using exponential sums approximations we construct a preconditioner that is compatible with the low-rank tensor product framework. The use of suitable preconditioning techniques yields a transformation of the operator equation in Hilbert space into a sequence space with Euclidean inner product, enabling rigorous error and rank control in the Euclidean metric.
\end{abstract}

\begin{keywords}
	isotropic radiative transfer, low-rank methods, iterative methods, compression, rank-control techniques, tensor product structure
\end{keywords}

\begin{AMS}
	65F10, 65F55, 65N22, 65N30
\end{AMS}

\section{Introduction}
Radiative transfer, that is, the phenomenon of energy transfer in the form of electromagnetic radiation including absorption and scattering, has many practical applications. Examples are medical imaging and tumor treatment \cite{AGH16, AS09, L97}, energy efficient generation of white light \cite{RPBSN18}, climate sciences \cite{E98, STS17}, and geosciences \cite{MWTLS17}. In its stationary form, the interactions between the radiation and the medium are modeled by the radiative transfer equation (RTE),
\begin{equation}
	\label{eq:RTE}
	\ss \cdot \nabla \varphi(\rr,\ss) + \sigma_t(\rr) \varphi(\rr,\ss) = \sigma_s(\rr)\int_S k(\ss\cdot\ss^{\p}) \varphi(\rr,\ss^{\p})\,d\ss^{\p} + q(\rr,\ss) \quad \text{on } D \times S
\end{equation}
where the specific intensity $ \varphi = \varphi(\rr,\ss) $ depends on the spatial coordinate $ \rr \in D\subset \RRR^d $ and on the direction $ \ss \in S $, with $S$ denoting the unit sphere in $\RRR^d$. The gradient appearing in \cref{eq:RTE} is taken with respect to $\rr$ only.
The physical properties of the medium covered by $D$ enter \eqref{eq:RTE} through the total attenuation (or transport) coefficient $ \sigma_t(\rr) = \sigma_a(\rr) + \sigma_s(\rr) $, which accounts for the absorption and scattering rates, respectively, and through the scattering kernel $ k(\ss\cdot\ss^{\p}) $, which describes the probability of scattering from direction $\ss^{\p}$ into direction $\ss$.
Internal sources of radiation are modeled by the function $ q(\rr,\ss) $.
We complement \cref{eq:RTE} by homogeneous inflow boundary conditions
\begin{align}
    \label{eq:RTE_bc}
    \varphi(\rr,\ss) = 0 \quad \text{for } (\rr,\ss) \in \d D_- \colonequals \{ (\rr,\ss)\in \d D \times S: \nn(\rr) \cdot \ss < 0 \}.
\end{align}
Here $ \nn(\rr) $ denotes the outward normal unit vector field for a point $ \rr \in \d D $.
We refer to \cite{C60, CZ67, DM79} for further details on the derivation of the radiative transfer equation.

Analytic solutions to the RTE exist for simple cases only, and in general, the solution $\varphi$ must be approximated by numerical methods.
Classical numerical methods, such as the spherical harmonics method and the discrete ordinates method \cite{DM79}, approximate $\varphi$ by functions of the form
\begin{align}\label{eq:approx}
    \sum_{j=1}^J \sum_{n=1}^N \varphi_{j,n} \psi_j(\rr) H_n(\ss),
\end{align}
where $\{\psi_j\}$ and $\{H_n\}$ represent systems of basis functions for $H^1(D)$ and $L^2(S)$, respectively. If $D$ and $S$ are partitioned by quasi-uniform triangulations with a mesh size parameter $h$, we obtain $J \sim h^{-d}$ and $N \sim h^{-d+1}$. Hence storage of the approximate solution is proportional to $\O(h^{-2d+1})$, which becomes prohibitively expensive already for moderate $h$.
We note that generally, good approximations require small $h$ due to the hyperbolic nature of \cref{eq:RTE} as well as the form of the boundary condition \cref{eq:RTE_bc}.

The dimensionality issue leading to this unfavorable scaling has been widely addressed in the literature on the RTE. In sparse tensor product methods \cite{WHS08,GS11a,GS11b}, the essential idea is  to replace the full set of indices $(j,n)$ with $1 \leq j\leq J$ and $1 \leq n \leq N$ in \cref{eq:approx} by a much smaller subset of indices $(j,n)$ with $1 \leq f(j,n)\leq J$ for some cut-off function $f$.
The general convergence analysis of these methods, however, requires strong smoothness assumptions on the solution, which are not guaranteed by the smoothing properties of \cref{eq:RTE}. Moreover, either scattering \cite{WHS08,GS11b} or the boundary condition \cref{eq:RTE_bc} are not fully taken into account \cite{GS11a}. Another class of methods constructs non-tensor product approximations.
For example, \cite{KL15} employs triangulations of the phase-space $D \times S$, where to each spatial element a separate triangulation of $S$ is constructed, and the solution is then approximated by a discontinuous Galerkin method. For a simplified geometry, a discontinuous Galerkin method for arbitrary triangulations of the phase-space has been developed in \cite{BBPS23}. In \cite{DGM18} the opposite approach is taken and to every discrete direction $\ss \in S$, a spatial grid is constructed using an adaptive Petrov-Galerkin method. While this approach has a strong theoretical foundation, the practical implementation is challenging and \rev{has been} done only for a small number of discrete directions $\ss$ and for transport-dominated problems.

\new{One approach for solving discretizations of stationary problems is to solve corresponding dynamical systems for very large times.
Step-and-truncate approaches employ time-stepping and rank truncation, which bears similarities with the Richardson iteration considered below; see  
\cite{GuoQiu2022,GuoQiu2024,GuoEmaQiu2024,Khaza2024} for application of such methodology in the context of Vlasov-Poisson systems.
The dynamical low-rank approach is based on numerically integrating low-rank matrix approximations of the discrete solution to the differential equation \cite{KochLubich2007}.}
For 
$d\in \{1,2\}$, a dynamical low-rank approximation method is investigated in \cite{PMF20}; a related rank-adaptive integrator for radiation therapy problems is developed in \cite{KS23}. \new{Let us refer to the recent review \cite{EKKMQ24} for an in-depth overview of approaches.}

In this paper we propose to address the dimensionality issue for the stationary RTE by developing a low-rank tensor product framework. 
In order to provide a rigorous analysis, we restrict ourselves to \emph{slab} or \emph{plane-parallel geometry} 
\cite{A98, DM79}.
In this context, $D = \RRR^2 \times (0,Z)$, with $Z>0$ denoting the thickness of the slab, and the specific intensity $ \varphi = \varphi(z,\mu) $ depends only on the \new{scalar} spatial coordinate $ z\in(0,Z) $ and on the (\new{scalar}) cosine $\mu\in(-1,1)$ of the polar angle of a unit vector $\ss\in S$. Moreover, we assume isotropic scattering $k(\ss \cdot \ss^{\p}) = {\rm const}$, the identity $q(z,\mu)=q(z,-\mu)$ for the source term, and that the scattering and absorption coefficients depend on $z$ only. In this setting, \cref{eq:RTE} becomes
\begin{align*}
    \mu\d_z \varphi(z,\mu) + \sigma_t(z) \varphi(z,\mu) &= \frac{\sigma_s(z) }{2}\int_{-1}^1 \varphi(z,\mu')\,d\mu' + q(z,\mu)\quad\text{for }(z,\mu)\in (0,Z)\times(-1,1),\\
    \varphi(0,\mu)&=0\quad\text{for }\mu\in (0,1),\\
    \varphi(Z,\mu)&=0\quad\text{for }\mu\in (-1,0).
\end{align*}
Denoting $\varphi^\pm(z,\mu):=\frac{1}{2}[\varphi(z,\mu)\pm\varphi(z,-\mu)]$ the even and odd parts of $\varphi$, orthogonality of even and odd functions yields $\mu\d_z\varphi^+ + \sigma_t \varphi^- =0$. Inserting the latter expression into the transport equation shows that $u:=\varphi^+$ satisfies the second order equation
\begin{equation}
    \label{eq:RTE_strong_slab_even_parity}
    -\d_z \left( \dfrac{\mu^2}{\sigma_t(z)}\d_z u(z,\mu) \right) + \sigma_t(z)u(z,\mu) = \sigma_s(z)\int_0^1 u(z, \mu^{\p})\,d\mu^{\p} + q(z,\mu)
\end{equation}
for $(z,\mu)$ in the slab $\Omega := (0,Z)\times(0,1)$, complemented with Robin boundary conditions
\begin{equation}
    \label{eq:Robin_bc}
    u(z,\mu) + \dfrac{\mu}{\sigma_t(z)}\d_n u(z,\mu) = g(z,\mu) \qquad \text{for } (z,\mu)\in \Big(\{0\} \times (0,1)\Big) \cup \Big(\{Z\} \times (0,1)\Big).
\end{equation}
Here, we introduced $g$ to allow for inhomogeneous boundary terms. The solution $\varphi$ is then given by $\varphi = u - \sigma_t^{-1}\mu \d_z u$, see \cite{ES12, BBPS23} for further details.

Using stable variational formulations as in \cite{ES12} turns \eqref{eq:RTE_strong_slab_even_parity}--\eqref{eq:Robin_bc} into a linear system of equations of dimension $JN\times JN$, where $J$ and $N$ are the respective dimensions of the approximation spaces for the spatial ($z$) and angular ($\mu$) variables. The corresponding system matrix is a short sum of Kronecker products, which can be stored and applied efficiently for many standard discretization schemes, such as discrete ordinates or spherical harmonics methods.
Our aim is to construct low-rank approximations of $u$, i.e., solution vectors of the form 
\begin{equation}
    \label{eq:low-rank-U}
    \UU = \sum_{k=1}^r \vv_k \otimes \ww_k \quad \text{with  $\vv_k \in \RRR^J$, $\ww_k \in \RRR^N$ for $ 1 \leq k \leq r$.} 
\end{equation}
The storage requirement for $\UU$ is thus $\O(r(J+N))$. 
If $r$ is much smaller than $J$ and $N$, then the representation \cref{eq:low-rank-U} has much lower memory requirements than the usual representation, which is of order $\O(JN)$. 
To solve the linear system for $\UU$ with computational complexity scaling like $\O(r^*(J+N))$, for some rank $r^*>0$, we consider suitable Richardson iterations, drawing inspiration from \cite{BD16,BS17} for similar constructions for high-dimensional elliptic problems.
More specifically, in order to prevent uncontrolled growth of the ranks in the iterative process, we will use
(i) a preconditioned iteration such that the iteration count is uniform in the size of the discretization parameters, and 
(ii) a suitable rank truncation technique.
We highlight that these two aspects are not independent of each other. This is because the error due to rank truncation, which is usually based on singular value decompositions applied to the current iterate, does not have a proper interpretation in terms of the underlying Hilbert space norm.
Instead, we apply left-right preconditioning to transform the Hilbert space structure to a Euclidean setting. 
In the Euclidean setting, errors due to rank truncation have a rigorous interpretation in problem-relevant Hilbert space norms.
To obtain an efficient algorithm, the preconditioner should be a short sum of Kronecker products. 
Using suitable change of bases for the approximation spaces, we show that the system matrix is spectrally equivalent to a symmetric positive definite matrix with Kronecker sum structure. For such matrices, inverse powers can be approximated efficiently by exponential sum approximations \cite{BM10, SY17, Y20}. As a result, we obtain preconditioners that are short sum of Kronecker products.
The combination of all these steps allows us to use results obtained in \cite{BS17} for elliptic problems to analyze the considered rank-controlled preconditioned iteration. 
In addition to studying the convergence of the iterative scheme, following \cite{BS17} we can establish upper and lower bounds for the distance between the Galerkin solution and the limit of the iteration, and since in general the decay behavior of the singular values of the fixed point is unknown, we provide an algorithm which adjusts the thresholding parameter \textit{a posteriori}. This is done based on the current iterate, retaining quasi-optimal ranks throughout the process. 

The outline of the manuscript is as follows. In \Cref{sec:Weak_formulation_and_Galerkin_approximation_of_the_even-parity_equations} we will recall the weak formulation of \cref{eq:RTE_strong_slab_even_parity}--\cref{eq:Robin_bc}, and a choice of bases which leads to linear systems with the corresponding representation of the system matrix as a short sum of Kronecker products.
In \Cref{sec:Preconditioning_via_exponential_sums_approximations} we discuss in detail the construction of the preconditioner.
In \Cref{sec:Transformed_linear_system_and_iterative_method} we then discuss different choices for rank truncation and state the main convergence results.
\Cref{sec:Numerical_experiments} contains numerical experiments for established discretization schemes, described in \Cref{sec:Numerical_realization}, such as truncated spherical harmonics expansions or piecewise constant approximations in $\mu$, which show the potential of the method.
We conclude the manuscript in \Cref{sec:conclusion} with some comments on possible future work.

\section{Weak formulation and Galerkin approximation of the even-parity equations}
\label{sec:Weak_formulation_and_Galerkin_approximation_of_the_even-parity_equations}

\subsection{Function spaces}
To introduce the weak formulation of \cref{eq:RTE_strong_slab_even_parity}--\cref{eq:Robin_bc}, let $L^2(\Omega)$ be the usual Hilbert space of square-integrable functions with inner product $\inner{\cdot}{\cdot}$ and induced norm $\normc$. Furthermore, we define the Hilbert space
\begin{align}
	W^2(\Omega) \colonequals \left\{ v\in L^2(\Omega):\, \mu\d_z v \in L^2(\Omega)\right\},
\end{align}
with norm induced by the inner product 
\begin{equation}
    \label{eq:inner_product_and_norm_in_W2}
    \inner{v}{w}_{W^2(\Omega)} \colonequals \inner{\mu\d_z v}{\mu\d_z w} + \inner{(1+\mu^2)v}{w}, \qquad v,w \in W^2(\Omega). 
\end{equation}
Since $1\leq 1+\mu^2\leq 2$, we note that $\|\cdot\|_{W^2(\Omega)}$ is equivalent to the graph norm given by $\normc_{W^2_{\rm G}(\Omega)}^2 \colonequals\|\cdot\|^2+\|\mu\d_z\cdot\|^2$. The notation $\normc_{\infty}$ is used for the $L^{\infty}$-norm in the $z$-variable. In order to incorporate boundary conditions, let us define the set $\d\Omega_- \colonequals \big(\{0\} \times (0,1)\big) \cup \big(\{Z\} \times (0,1)\big)$ and denote $L^2(\d\Omega_-;\mu)$ the corresponding weighted Hilbert space of square-integrable functions with inner product denoted by $(\mu\cdot,\cdot)_{\d\Omega_-}$. 
We have the following trace lemma \cite[Theorem 2.8, p.55]{A98}.
\begin{lemma}
    \label{lem:trace}
	There exists a bounded linear operator $\tau: W^2(\Omega) \to L^2(\d\Omega_-;\mu)$, which satisfies 
	\begin{equation}
        \label{eq: trace inequality}
	    \norm{\tau v}_{L^2(\d\Omega_-;\mu)} \leq C_{\tr} \norm{v}_{W^2_G(\Omega)} \leq C_{\tr} \norm{v}_{W^2(\Omega)}
	\end{equation}
	for any $v\in W^2(\Omega)$ and $C_{\tr} \colonequals {2}/\sqrt{1-\exp(-2Z)}$. 
	Moreover, $\tau v$ can be identified with $\tau_0 v=v(0,\cdot)$ and $\tau_Z v=v(Z,\cdot)$ if $v\in C^0(\overline{\Omega})$.
\end{lemma}

\subsection{Weak formulation}
Following the usual procedure (see, e.g., \cite{PS20}) in our context,
we can rewrite \cref{eq:RTE_strong_slab_even_parity}--\cref{eq:Robin_bc} in weak form as follows: find $u \in W^2(\Omega)$ such that
\begin{align}
    \label{eq: weak formulation}
	\au(u,v) = b(v) \quad\text{for all } v\in W^2(\Omega),
\end{align}
with bilinear form $\au: W^2(\Omega) \times W^2(\Omega) \to \RRR$ and linear form $b: W^2(\Omega) \to \RRR$ defined by
\begin{align}
    \label{eq:bilinear_form_aE}
	\au(u,v) & \colonequals \inner{\dfrac{\mu^2}{\sigma_t}\d_z u}{\d_z v} + \inner{\sigma_t u}{v} - \inner{\sigma_s \S u}{v} + (\mu\tau u,\tau v)_{\d\Omega_-}, \\
	b(v) & \colonequals \inner{q}{v} + \inner{\mu g}{v}_{\d\Omega_-}.
\end{align}
Here $\S :L^2(0,1) \to L^2(0,1)$ denotes the $\mu$-averaging (scattering) operator defined by
\begin{align}
    \label{eq:scattering_operator}
	(\S v)(\mu) \colonequals \int_0^1 v(\mu') d\mu'.
\end{align}
For $v\in L^2(\Omega)$, we also write $(\S v)(z,\mu):=(\S v(z,\cdot))(\mu)$. The following hypotheses on the source terms and the optical parameters are natural in practical applications.
\begin{assumption}
    \label{ass:1}
	(i) The source terms satisfy $q\in L^2(\Omega)$ and $g\in L^2(\partial\Omega_-;\mu)$.\\\noindent
	(ii) The optical parameters satisfy $0\leq \sigma_s,\sigma_t \in L^\infty(0,Z)$ and there exists a constant $c_0>0$ such that $c_0\leq \sigma_t-\sigma_s$. 
\end{assumption}
Well-posedness of \cref{eq: weak formulation} follows from standard arguments, see, e.g., \cite{PS20}. We give a proof to make the constants explicit.
\begin{lemma}
    \label{lem:spectral_equivalence_Au_Pu}
	If \Cref{ass:1}(ii) holds, then any $v\in W^2(\Omega)$ satisfies
	\begin{equation*}
	    \gamma_1 \norm{v}^2_{W^2(\Omega)} \leq \au(v,v) \leq \gamma_2 \norm{v}^2_{W^2(\Omega)}
	\end{equation*}
	with $\gamma_1 \colonequals \min\{\|\sigma_t\|_{{\infty}}^{-1},\, c_0/2\}$ and $\gamma_2 \colonequals \max\{ \|\sigma_t^{-1}\|_\infty,\, \|\sigma_t\|_\infty,\, C_{\tr}^2\}$, where the trace constant $C_{\tr}$ is defined in \Cref{lem:trace}.
\end{lemma}
\begin{proof}
    Using that $\sigma_t-\sigma_s\geq c_0$ by \Cref{ass:1}(ii), we obtain that 
    \begin{align*}
        \au(v,v) & = \inner{\frac{\mu^2}{\sigma_t}\d_z v}{\d_z v} + \inner{\sigma_t v}{v} - \inner{\sigma_s \S v}{v} + \inner{\mu\tau v}{\tau v}_{\partial\Omega_-} \\
        & \geq \dfrac{1}{\norm{\sigma_t}_{\infty}}\norm{\mu \d_z v}^2 + c_0\norm{v}^2 \geq \dfrac{1}{\norm{\sigma_t}_{\infty}}\norm{\mu \d_z v}^2 + \dfrac{c_0}{2} \norm{\sqrt{1+\mu^2}v}^2,
    \end{align*}
    which shows the first inequality in the statement. For the second inequality, we employ \Cref{lem:trace} and obtain that
    \begin{align*}
        \au(v,v) & = \inner{\frac{\mu^2}{\sigma_t}\d_z v}{\d_z v} + \inner{\sigma_t v}{v} - \inner{\sigma_s \S v}{v} + \inner{\mu\tau v}{\tau v}_{\partial\Omega_-} \\
        & \leq \norm{1/\sigma_t}_{\infty} \norm{\mu \d_z v}^2 + \norm{\sigma_t}_{\infty} \norm{\sqrt{1+\mu^2}v}^2 + C_{\tr}^2 \norm{v}_{W^2(\Omega)}^2.
    \end{align*}
\end{proof}
\begin{lemma}
    \label{lem: well-posedness of the weak formulation}
	Let \Cref{ass:1} hold. Then there exists a unique $u \in W^2(\Omega)$ satisfying \cref{eq: weak formulation}.
\end{lemma}
\begin{remark}\label{rem:operator_formulation}
	The variational formulation \eqref{eq: weak formulation} can be restated as an operator equation with differential operator
	$\Au: W^2(\Omega) \to W^2(\Omega)^{\ast}$ defined by the relation
	\begin{equation*}
	    \duality{\Au v}{w}_\ast = \au(v,w) \quad \text{for all }v,w \in W^2(\Omega).
	\end{equation*}
	Here, $W^2(\Omega)^{\ast}$ denotes the dual space of $W^2(\Omega)$, and $\langle\cdot,\cdot\rangle_\ast$ their duality pairing.
	Note that $\Au$ is a short-sum of tensor products, i.e.
	\begin{equation}
	    \label{eq:differential_operator_Au}
	    \Au = \mu^2 \otimes \d_z^{\p}\left( \Sigma_t^{-1}\d_z \right) + \I_{\mu} \otimes \Sigma_t - \S \otimes \Sigma_s + \mu \otimes (\tau_0^{\p}\tau_0 + \tau^{\p}_Z\tau_Z),
	\end{equation}
	where $\d_z^{\p}:L^2(0,Z) \to H^1(0,Z)'$  is defined by $\d_z^{\p}: u \mapsto (u, \d_z\cdot)_{L^2(0,Z)}$.
	Here $H^1(0,Z)$ denotes the usual Sobolev space.
	Moreover, $\Sigma_t, \Sigma_s$ are the multiplication operators $L^2(0,Z)\to L^2(0,Z)$ defined by $u \mapsto \sigma_t u$ and $u \mapsto\sigma_s u$.
    \Cref{lem:spectral_equivalence_Au_Pu} can be interpreted as the spectral equivalence of $\Au$ and the operator $\Pu: W^2(\Omega) \to W^2(\Omega)^{\ast}$ defined by
    \begin{equation*}
        \duality{\Pu v}{w}_\ast = \inner{v}{w}_{W^2(\Omega)} \qquad \text{for all } v,w\in W^2(\Omega).
    \end{equation*}
    $\Pu$ exhibits a sum of tensor products structure
    \begin{equation}
        \label{eq: Preconditioning operator}
        \Pu = \mu^2 \otimes \d_z^{\p}\d_z  + (1+\mu^2) \otimes \I_z= \mu^2 \otimes(\d_z^{\p}\d_z + \I_z) + \I_{\mu} \otimes \I_z,
    \end{equation}
    with $\I_i$, $i=z, \mu$, denoting the identity operators in $L^2(0,Z)$ and $L^2(0,1)$, respectively. 
    The second representation will be useful below, because $\d_z^{\p}\d_z + \I_z:H^1(0,Z)\to H^1(0,Z)^{\p}$ is boundedly invertible, and thus the
    operator $\Pu$ is a candidate for a preconditioner for $\Au$.
\end{remark}

\subsection{Galerkin approximation}
\label{sec:Galerkin_approximation}
We define the tensor product approximation space
\begin{equation}
    \label{eq:approximation_space}
    W^2_{J,N} = W^z_J \otimes W^\mu_N = \left\{ v_{J,N} \in W^2(\Omega):\, v_{J,N}(z,\mu) = \sum_{j=1}^J \sum_{n=1}^N v_{j,n} \psi_j(z) H_n(\mu) \right\} \subset W^2(\Omega),
\end{equation}
where $\{\psi_j(z)\}_{j=1}^J$ and $\{H_n(\mu)\}_{n=1}^N$ denote the sets of basis functions in space and angle, respectively. The Galerkin approximation of \cref{eq:scattering_operator} then reads: find $u_{J,N}\in W^2_{J,N}$ such that
\begin{align}\label{eq:Galerkin}
    \au(u_{J,N},v_{J,N})= b(v_{J,N}) \quad \text{for all }v_{J,N} \in W^2_{J,N}.
\end{align}
Since the bilinear form $\au$ is symmetric and coercive on $W^2(\Omega)$ and $W^2_{J,N}\subset W^2(\Omega)$, standard arguments imply the following result.
\begin{lemma}
    Let \Cref{ass:1} hold. Then there exists a unique solution $u_{J,N} \in W^2_{J,N}$ of \cref{eq:Galerkin}. Moreover, the  quasi-best approximation error estimate
    \begin{align*}
        \norm{u-u_{J,N}}_{W^2(\Omega)} \leq \dfrac{\gamma_2}{\gamma_1} \inf_{v_{J,N} \in W^2_{J,N}} \norm{u-v_{J,N}}_{W^2(\Omega)}
    \end{align*}
    holds, with $\gamma_1$, $\gamma_2$ from \Cref{lem:spectral_equivalence_Au_Pu}.
\end{lemma}
As usual, we identify a function $u_{J,N} \in W^2_{J,N}$ with its coefficients $\{u_{j,n}\}$, $j=1,\dots, J$ and $n = 1,\dots,N$, which we collect in a matrix $\UU \in \RRR^{J \times N}$. Accordingly, we denote by $\uu = \vec(\UU) \in \RRR^{JN}$ the vectorization of $\UU$, which we obtain by stacking the columns of $\UU$ into a column vector. We will also write $\UU = \mat(\uu)$ for the reverse operation (matricization).
The Galerkin approximation \cref{eq:Galerkin} is equivalent to the linear system
\begin{align}
    \label{eq:linear_system}
    \AAAu\uu = \ffu
\end{align}
with matrix $\AAAu \in \RRR^{JN\times JN}$. Let us give some details on the structure of $\AAAu$.

\subsection{Structure of system matrix}
\label{sec:Structure_of_system_matrix}
For coefficient functions $\nu(z)$ and $\eta(\mu)$, we define the matrices $\DD(\nu)$, $\MM_z$, $\BB$, $\MM_\mu$ and $\SS$ by their entries for  $i,j\in \{1,\ldots,J\}$ and $m,n \in \{1,\ldots,N\}$ as
\begin{align*}
	[\DD(\nu)]_{i,j} & \colonequals \int_0^Z \nu(z) \d_z \psi_j(z) \d_z \psi_i(z) \, dz, \\
	[\MM_z(\nu)]_{i,j} & \colonequals \int_0^Z \nu(z) \psi_j(z) \psi_i(z) \, dz, \\
	[\BB]_{i,j} & \colonequals \psi_j(Z) \psi_i(Z) + \psi_j(0) \psi_i(0) \\
	[\MM_{\mu}(\eta)]_{m,n} & \colonequals \int_0^1 \eta(\mu) H_m(\mu) H_n(\mu) \, d\mu, \\
	[\SS]_{m,n} & \colonequals \int_0^1 (\S H_n) H_m \, d\mu.
\end{align*}
We write $\DD = \DD(1)$ and $\MM_{l} = \MM_{l}(1)$ for $l=z,\mu$. Using Fubini's theorem, we obtain the identity
\begin{align}
    \label{eq:def_A}
	\AAAu = \MM_{\mu}(\mu^2) \otimes \DD(\sigma_t^{-1}) + \MM_{\mu} \otimes \MM_z(\sigma_t) - \SS \otimes \MM_z(\sigma_s) + \MM_{\mu}(\mu) \otimes \BB.
\end{align}
Thus $\AAAu$ is a short sum of Kronecker products, cf. 
\cref{eq:differential_operator_Au}. If $\UU$ has the form \cref{eq:low-rank-U}, we also apply $\AAAu$ to matrices and write $\AAAu\UU:=\mat(\AAAu\vec(\UU))$ with slight abuse of notation.
\subsection{Matrix representation for the Riesz map}
\label{sec:Matrix_representation_for_the_Riesz_map}
For notational convenience, let us make the assumption that $\{\psi_j\}_{j=1}^J$ is an orthonormal basis in $H^1(0,Z)$ and that $\{H_n\}_{n=1}^N$ is orthonormal in $L^2(0,1)$; see \Cref{rem:change_of_basis} for the general case. Similar to representing the even-parity bilinear form $a_E$ by $\mathbb{E}$, we obtain the matrix representation $\mathbb{J}$ of $\mathcal{J}$ defined in \cref{eq: Preconditioning operator} as follows:
\begin{align}
    \label{eq:matrix_J}
	\PPPu = \MM_\mu(\mu^2)\otimes \II_z + \II_\mu\otimes \MM_z.
\end{align}
The particular structure on the right-hand side of \cref{eq:matrix_J} is called a Kronecker sum.
As an immediate consequence of \Cref{lem:spectral_equivalence_Au_Pu}, we have the spectral equivalence of $\AAAu$ and $\PPPu$, i.e.,
\begin{align}
    \label{eq: spectral equivalence Au and Pu}
    \gamma_1 \duality{\JJJ \vv}{\vv} \leq \duality{\AAAu \vv}{\vv} \leq \gamma_2 \duality{\JJJ \vv}{\vv} \quad \text{for all } \vv \in \RRR^{JN}.
\end{align}
where $\duality{\cdot}{\cdot}$ denotes here the standard Euclidean inner product. Introducing the new variable $\ww = \JJJ^{1/2} \uu$, \cref{eq: spectral equivalence Au and Pu} becomes
\begin{equation*}
    \gamma_1 \norm{\ww}^2_2 \leq \langle\EEE\uu, \uu\rangle \leq \gamma_2 \norm{\ww}^2_2,
\end{equation*}
for all $u_{J,N} \in W^2_{J,N}$ with coefficient vector $\uu\in\RRR^{JN}$. Using the matricizations, we get 
\begin{equation*}
	\gamma_1 \norm{\WW}^2_{\frob} \leq \norm{\UU}_{\EEE}^2 \leq \gamma_2 \norm{\WW}^2_{\frob},
\end{equation*}
i.e., we have control over the energy norm of $\UU$ in terms of the Frobenius norm of $\WW$. 
In other words, we can control the Hilbert space norm of $u_{J,N}$ in terms of the Euclidean norm of $\ww$.

\begin{remark}[\textbf{Change of basis}]
    \label{rem:change_of_basis}
    In case $\{\psi_j\}$ or $\{H_n\}$  are not orthonormal, we consider Cholesky factorizations of the Gram matrices, $\DD+\MM_z=\TT_z \TT_z^T$ and $\MM_\mu=\TT_\mu \TT_\mu^T$, and observe that the representation of $\J$ in that basis can be rewritten as
    \begin{align*}
	   \MM_{\mu}(\mu^2) \otimes (\DD + \MM_z) + \MM_{\mu} \otimes \MM_z = 
	   (\Tm \otimes \Tz)\left(\hat{\MM}_{\mu}(\mu^2)\otimes \II_z + \II_{\mu} \otimes \hat{\MM}_z\right)(\Tm\otimes\Tz)^T.
    \end{align*}
    Therefore, up to the change of basis implied by $(\Tm\otimes\Tz)$, the structure of the matrix representation of $\J$ is as in \cref{eq:matrix_J}.
\end{remark}

\section{Preconditioning via exponential sums approximations}
\label{sec:Preconditioning_via_exponential_sums_approximations}
We next construct approximations of inverse powers of $\PPPu$ using approximations of functions $\Phi_\beta(t):=t^{-\beta}$, $\beta>0$, by a finite series of exponentials in combination with spectral calculus.
As described in \cite{SY17} (see also \cite{BM10} and the overview in \cite{B23}), such approximations can be obtained by discretizing integral representations of $\Phi_{\beta}(t)$ based on the Gamma function 
\begin{equation*}
    \Gamma(z) = \int_0^{+\infty} e^{-\tau} \tau^{z-1}\,d\tau, \quad \mathrm{Re}(z)>0. 
\end{equation*}
Observe that in the definition above, $z$ denotes a generic complex number. As shown in \cite[Lemma 5.1]{SY17}, we have for $t>0$ and $\mathrm{Re}(z)>0$ the identity
\begin{equation*}
    \Gamma(z) = t^z \int_{\RRR} \exp(-t e^\tau + z\tau)\,d\tau.
\end{equation*}
Therefore, we have, for any real $\beta >0$, the representation
\begin{equation}
    \label{eq: integral representation of t^(-beta)}    
	\Phi_\beta(t) = \dfrac{1}{\Gamma(\beta)} \int_{\RRR} \exp(-t e^\tau+\beta \tau)\,d\tau,
\end{equation}
with rapidly decaying integrand for $\tau\to\pm\infty$. Discretizing \cref{eq: integral representation of t^(-beta)} via the trapezoidal rule with step size $h$ on the whole real line, and truncating the infinite series, we obtain the approximation
\begin{equation}
    \label{eq:approximation_property}
    \Phi_\beta(t) \approx \Psi_\beta(t)\colonequals\dfrac{h}{\Gamma(\beta)} \sum_{i=-i_1}^{i_2} \alpha_i(\beta) \exp(-\rho_i t)
\end{equation}
for $t>0$, with $i_1, i_2\in \NNN$, $\alpha_i(\beta) = e^{\beta ih}$ and $\rho_i = \alpha_i(1)$. Setting $\beta = 1/2$, and choosing $h$, $i_1$ and $i_2$ according to a tolerance $\epsilon>0$ such that the following relative error bound holds
\begin{equation}
    \label{eq:relative_bound_on_exponential_sums_approximation}
    \vert \Phi_{1/2}(t) - \Psi_{1/2}(t) \vert \leq \epsilon \vert \Phi_{1/2}(t) \vert \quad \text{for } 0 < \lambda \leq t \leq \Lambda,
\end{equation}
where $\lambda, \Lambda$ denote the lower and upper bounds on the spectrum of $\PPPu$ defined in \cref{eq:matrix_J}, respectively, spectral calculus yields that 
\begin{equation}
    \label{eq: error bound on JJ half with norms}
 \norm{(\Phi_{1/2}(\PPPu)-\Psi_{1/2}(\PPPu))\vv}_2 \leq \epsilon \norm{\Phi_{1/2}(\PPPu)\vv}_2\quad\text{for all }\vv\in \RRR^{JN}.
\end{equation}
\begin{remark}
    Using the structure of $\PPPu$, we observe that each eigenvalue of $\PPPu$ is given by a sum of an eigenvalue of $\MM_\mu(\mu^2)$ and $\MM_z$. In particular $\Lambda = \norm{\MM_\mu(\mu^2)}_2 + \norm{\MM_z}_2$.
\end{remark}

Given the exponentiation property of the Kronecker sum, i.e., 
\begin{equation}
\label{eq:exp_formula_matrices}
	\exp(\AA \otimes \II + \II \otimes \BB ) = \exp(\AA)\otimes \exp(\BB),
\end{equation}
valid for general matrices $\AA$ and $\BB$, we can define our actual preconditioner using the representation
\begin{equation}
    \label{eq:definition_P^(-1/2)}
    \PPPw^{-1/2} := \Psi_{1/2}(\PPPu) = \dfrac{h}{\sqrt{\pi}} \sum_{i = -i_1}^{i_2} \alpha_i(1/2) \exp(-\rho_i \MM_{\mu}(\mu^2)) \otimes \exp(-\rho_i\MM_z). 
\end{equation}
Hence, $\PPPw^{-1/2}$ is a sum of Kronecker products consisting of $i_1+i_2+1$ terms. The following result, specializing \cite[Proposition~4.7]{B23}, ensures that $\PPPw$ is a suitable preconditioner for $\AAAu$.
\begin{lemma}
    \label{lem:spectral_equivalence_exp_sum}
    Let $\epsilon < 1$. If \cref{eq: error bound on JJ half with norms} holds, then the operator $\PPPw \colonequals \Psi_{1/2}(\PPPu)^{-2}$ is spectrally equivalent to $\AAAu$, i.e., for all $\vv \in \RRR^{JN}$ we have
    \begin{equation}
        \label{eq: equivalence between A and P matrices}
        \gamma_1^\epsilon \|\ww\|_2^2 \leq \duality{\PPPw^{-1/2}\AAAu\PPPw^{-1/2} \ww}{\ww} \leq \gamma_2^\epsilon \|\ww\|_2^2,
    \end{equation}
	where $\gamma_1^\epsilon=(1-\epsilon)^2\gamma_1$ and $\gamma_2^\epsilon=(1+\epsilon)^2\gamma_2$, and $\gamma_1, \gamma_2$ are as in \Cref{lem:spectral_equivalence_Au_Pu}.
\end{lemma}
\begin{proof}
Setting $\vv=\JJJ^{1/2}\ww$ in \cref{eq: error bound on JJ half with norms} yields
\begin{align*}
    \|(\II-\PPPw^{-1/2}\JJJ^{1/2})\ww\|_2 \leq \epsilon\|\ww\|_2.
\end{align*}
Then, for $\hat\AAAu= \PPPu^{-1/2} \AAAu \PPPu^{-1/2}$, we have that
$\langle \PPPw^{-1/2} \AAAu \PPPw^{-1/2}\ww,\ww\rangle = \langle \hat\AAAu \PPPu^{1/2} \PPPw^{-1/2}\ww, \PPPu^{1/2}\PPPw^{-1/2}\ww \rangle$.
Hence, using the spectral equivalence \cref{eq: spectral equivalence Au and Pu}, we obtain that
\begin{align*}
	\gamma_1 \|\PPPu^{1/2} \PPPw^{-1/2}\ww\|_2^2\leq \langle \PPPw^{-1/2} \AAAu \PPPw^{-1/2}\ww, \ww\rangle \leq \gamma_2 \|\PPPu^{1/2}\PPPw^{-1/2}\ww\|_2^2.
\end{align*}
Moreover, by the triangle inequality ($\III$ is the identity matrix in $\RRR^{JN\times JN})$,
\begin{align*}
    \|\PPPu^{1/2} \PPPw^{-1/2}\ww\|_2 &\leq \|(\III-\PPPw^{-1/2}\PPPu^{1/2})\ww\|_2  + \|\ww\|_2 \leq (1+\epsilon) \|\ww\|_2,\\
    \|\PPPu^{1/2} \PPPw^{-1/2}\ww\|_2 & \geq \left| \|(\III-\PPPw^{-1/2} \PPPu^{1/2})\ww\|_2  - \|\ww\|_2 \right| \geq (1-\epsilon) \|\ww\|_2,
\end{align*}
concluding the proof.
\end{proof}

\begin{remark}
    \label{rem:choice_of_parameters_in_exp_sum}
	Let us comment on the choices of the parameters $h$, $i_1$, $i_2$ in \cref{eq:approximation_property} to ensure the bound \cref{eq:relative_bound_on_exponential_sums_approximation}.
    As shown in \cite[Theorem 3]{BM10}, it suffices to choose $h=h(\epsilon)$ such that 
    \begin{equation}
        h \leq \dfrac{2\pi}{\ln(3) +\beta |\ln \cos 1| + |\ln(\epsilon)|}.
    \end{equation}
    Hence, $h$ depends only logarithmically on $\epsilon$, but not on $\Lambda/\lambda$.
    We note that we require only the spectral equivalence \cref{eq: equivalence between A and P matrices}, that is, in practice it suffices to choose a fixed $\epsilon>0$ (say, $\epsilon=1/10$).
    Performing the change of variable $s = \ln t$, \cref{eq:relative_bound_on_exponential_sums_approximation} amounts to look for $i_1$ and $i_2$ such that 
    \begin{equation}
        \label{eq:logarithmic_dependance_eigenvalues}
        1-\epsilon \leq \sum_{i=-i_1}^{i_2} \exp\bigl(-e^{ih+s}+ \textstyle\frac{1}{2}\displaystyle(ih+s)\bigr)  \leq 1+\epsilon \quad\text{for all } \ln(\lambda) \leq s \leq \ln(\Lambda).
    \end{equation}
    Hence, $i_1$, $i_2$ depend logarithmically on $\lambda$ and $\Lambda$.
\end{remark}

\section{Transformed linear system and iterative method}
\label{sec:Transformed_linear_system_and_iterative_method}
Using the definition of the preconditioner in \cref{eq:definition_P^(-1/2)}, we will from now on define $\AAAw = \PPPw^{-1/2}\AAAu \PPPw^{-1/2} $ and consider the left-right preconditioned counterpart of \cref{eq:linear_system}
\begin{align}
    \label{eq:linear_system_transformed}
    \AAAw \ww = \ffw,
\end{align}
where $\ffw = \PPPw^{-1/2}\ffu$ and $\ww = \PPPw^{1/2}\uu$. 
The next result, which follows from \Cref{lem:spectral_equivalence_exp_sum}, shows that a Richardson iteration for the system \cref{eq:linear_system_transformed} converges for a suitable choice of the step-size parameter. 
\begin{theorem}
    \label{thm:convergence_standard_richardson}
	Let $\epsilon<1$ and $\gamma_2^\epsilon$ as in \Cref{lem:spectral_equivalence_exp_sum}. For any step-size parameter $\omega \in (0, 2/\gamma_2^\e)$ and any $\ww_0 \in \RRR^{JN}$ the iteration
	 \begin{align*}
	     \ww_{k+1} = \ww_k - \omega(\AAAw \ww_k - \ffw), \qquad k \geq 0,
	 \end{align*}
	 converges linearly to the solution $\ww$ of \cref{eq:linear_system_transformed}. In particular, for $\omega^*=2/(\gamma_1^\epsilon+\gamma_2^\epsilon)$ we have that
	 \begin{align*}
	     \norm{\ww - \ww_k}_2 \leq \rho^k \norm{\ww - \ww_0}_2 \qquad \text{with }\rho = \dfrac{\gamma_2^\epsilon-\gamma_1^\epsilon}{\gamma_2^\epsilon+\gamma_1^\epsilon}.
	 \end{align*}
\end{theorem}
Since $\AAAw$ is composed of operators that are short sums of Kronecker products, the memory requirement of storing $\AAAw$ is $\O((J+N)(i_1+i_2+1))$.

\begin{remark}
    We remark that the contraction rate in \Cref{thm:convergence_standard_richardson} does not depend on $J$ and $N$ but only on intrinsic properties of the problem considered. Furthermore both the optimal step-size parameter $\omega$ and the contraction rate $\rho$ can be computed explicitly. 
\end{remark}

\subsection{Rank controlled iteration}
\label{sec:rank_controlled_iteration}
From the representations of \cref{eq:definition_P^(-1/2)} for $\PPPw^{-1/2}$ and \cref{eq:def_A} for $\AAAu$, we observe that the ranks of $\WW_{k}=\mat(\ww_k)$ defined in \Cref{thm:convergence_rank_richardson} will grow in each iteration by $4(i_1+i_2+1)^2$. Therefore, we next consider rank truncation based on the singular value decomposition
\begin{align}
    \label{eq:SVD_W}
    \WW = \sum_{k=1}^R \sigma_k \uu_k\otimes\vv_k,
\end{align}
with $R\leq \min\{J,N\}$ denoting the rank of $\WW$. Here $\sigma_k>0$ are the singular values of $\WW$, which we assume to be ordered non-increasingly, and $\{\uu_k\}$ form an orthonormal basis of the range of $\WW$, while $\{\vv_k\}$ form an orthonormal basis for the orthogonal complement of the kernel of $\WW$.

For $\delta\geq 0$, consider the soft thresholding function $s_\delta$, which is defined for real numbers $t$ by
\begin{align*}
    s_\delta(t)={\rm sgn}(t)\max\{0, |t|-\delta\}.
\end{align*}
The soft thresholding operator $\S_\delta:\RRR^{J \times N}\to \RRR^{J\times N}$ is then defined via
\begin{align}\label{eq:def_soft_thresholding}
    \S_\delta(\WW) = \sum_{k=1}^R s_\delta(\sigma_k) \uu_k\otimes\vv_k.
\end{align}

The next result recalls the non-expansiveness of the soft thresholding operator with respect to the Frobenius norm; see, e.g., \cite[Proposition~3.2]{BS17} for a proof.
\begin{proposition}
    For any $\VV,\WW \in \RRR^{J \times N}$ and $\delta \geq 0$, the operator $\S_\delta$ defined in \cref{eq:def_soft_thresholding} satisfies
    $\norm{\S_\delta(\VV) - \S_\delta(\WW)}_{\frob} \leq \norm{\VV-\WW}_{\frob}$.
\end{proposition}
Since the composition of a non-expansive map with a contraction is a contraction, with \cite[Lemma 4.1]{BS17} applied to our setting, we obtain the following result.
\begin{theorem}
    \label{thm:convergence_rank_richardson}
    For any $\delta>0$, $\WW_0 = \mat{(\ww_0)} \in \RRR^{J\times N}$, and $\omega$, $\rho$ as in \cref{thm:convergence_standard_richardson}, the iteration
    \begin{equation}
        \label{eq:thresholded_iteration}
        \WW_{k+1} = \Sd(\WW_k - \omega (\AAA \WW_k-\FFw))
    \end{equation}
    has a unique fixed point $\WW^\delta$ satisfying
    \begin{align*}
        \norm{\WW_k- \WW^\delta}_{\frob} \leq \rho^k \norm{\WW_0 - \WW^\delta}_{\frob}.  
    \end{align*}
    Moreover, the distance between $\WW^{\delta}$ and the matricization $\WW^{\ast}$ of the fixed point of the non-thresholded version of \cref{eq:thresholded_iteration} satisfies
    \begin{equation}
        (1+\rho)^{-1} \norm{\S_\delta(\WW^{\ast}) - \WW^{\ast}}_{\frob} \leq \norm{\WW^\delta - \WW^{\ast}}_{\frob} \leq (1-\rho)^{-1} \norm{\S_\delta(\WW^{\ast}) - \WW^{\ast}}_{\frob}.
    \end{equation}
\end{theorem}
\Cref{thm:convergence_rank_richardson} shows that the error $\norm{\WW^\delta - \WW^{\ast}}_{\frob}$ in the limit $\WW^\delta$ of the thresholded iteration is proportional to the error of thresholding the exact solution, which will be small if $\delta$ is chosen sufficiently small in relation to the decay of the singular values of $\WW^{\ast}$. 

In general, instead of using a fixed thresholding parameter $\delta$, it is useful to consider a sequence $\delta_k\to 0$ such that the value of the parameter is decreased or maintained adaptively in each iteration, balancing the need to speed up convergence with the need to keep $\WW^{\delta_k}$ as close as possible to $\WW^{\ast}$. If the decay behavior of the singular values of $\WW^{\ast}$ is known, then \textit{a priori} choice rules for the sequence of thresholding parameters are proposed and analyzed in \cite{BS17}. In the more practical situation where such a decay behavior is not known, or only suboptimal estimates for the decay are available, \textit{a posteriori} choice rules as given in \cite[Section~5]{BS17} are more useful. In our case, we implement such dynamical method in \Cref{alg:dynamical_thresholding}.
\begin{algorithm}
    \caption{ $\WW^\e={\rm STSolve}(\AAAw,\FFw;\e)$ adapted from \cite[Algorithm~2]{BS17}.
    \label{alg:dynamical_thresholding}}
    \begin{algorithmic}[1]
        \REQUIRE $\delta_0\geq \omega \norm{\FF}_{\frob}$, $\nu,\theta\in (0,1)$, $\gamma_1^\epsilon,\gamma_2^\epsilon, \omega, \rho$ as in \Cref{thm:convergence_standard_richardson}
        \STATE $\WW_0 \colonequals 0$, $\RR_0 \colonequals - \FF$, $k \colonequals 0$
        \WHILE{$ \norm{\RR_k}_{\frob} > \gamma_1^\epsilon \e$}
            \STATE $\WW_{k+1} \colonequals \S_{\delta_k}(\WW_k- \omega \RR_k)$
            \STATE $\RR_{k+1} \colonequals \AAAw\WW_{k+1} - \FF$
            \IF{$\norm{\WW_{k+1}-\WW_{k}}_{\frob} \leq \frac{(1-\rho)\nu}{\gamma_2^\e \rho} \norm{\RR_{k+1}}_{\frob}$}
                \STATE $\delta_{k+1}\colonequals \theta \delta_k$
            \ELSE
                \STATE $\delta_{k+1}\colonequals \delta_k$
            \ENDIF
            \STATE $k\leftarrow k+1$
        \ENDWHILE
        \STATE $\WW^\e\colonequals \WW_k$
    \end{algorithmic}
\end{algorithm}

In order to discuss the convergence behavior of \Cref{alg:dynamical_thresholding}, we need some notation. By density of the approximation spaces for $J,N\to \infty$, we may think of infinite sequences of singular values now.
For $\e>0$, the truncated singular value decomposition
\begin{equation}
    \label{eq:truncated_SVD}
	\WW^{r^*} = \sum_{k=1}^{r^*} \sigma_k \uu_k \otimes \vv_k
\end{equation}
yields the optimal, i.e., smallest, rank $r^*=r^*(\e)$ which satisfies $\|\WW^{r^*}-\WW\|_{\frob}\leq \e$. We say that a family $\{\WW^\e\}_{\e>0}$ of approximations for $\WW$ has quasi-optimal ranks if there exists a constant $C>0$ such that $\norm{\WW - \WW^\e}_\frob \leq C \e$ and if the rank of $\WW^\e$ is bounded by $ C r^*(\e)$.

For $p>0$, the $\ell^{p,\infty}$-norm of the sequence $\boldsymbol\sigma =\{\sigma_k\}$ of singular values is defined by
\begin{align*}
    \norm{\boldsymbol\sigma}_{\ell^{p,\infty}} \colonequals \sup_{n \in \NNN} n^{\frac{1}{p}} \sigma_n,
\end{align*}
provided the expression on the right-hand side is finite. The corresponding sequence space satisfies $\ell^{p,\infty}\subset \ell^2$ for $0<p<2$, and $\boldsymbol\sigma \in \ell^{p,\infty}$ if and only if there exists $C>0$ such that
\begin{align*}
    \sqrt{\sum_{k>n} |\sigma_k|^2} \leq C n^{-s}\quad\text{for } s\colonequals \frac{1}{p}-\frac{1}{2},
\end{align*}
see, e.g., \cite{DV98}. Therefore, $\boldsymbol\sigma\in \ell^{p,\infty}$ for some $0<p<2$ ensures algebraic decay of the singular values and the quasi-best ranks scale like $C^{1/s} \e^{-1/s}$.

The next result, whose proof carries over from the proof of \cite[Theorem 5.1]{BS17}, ensures that the iterates produced by the algorithm are of quasi-optimal ranks.
\begin{theorem}
    \label{thm:st_richardson}
    For any $\e>0$, \Cref{alg:dynamical_thresholding} produces an approximation $\WW^\e$ to $\WW^\ast$ with $\norm{\WW^{\e}-\WW^{\ast}}_{\frob} \leq \e$ in finitely many steps. Moreover,
    \begin{enumerate}[(i)]
        \item if $\boldsymbol\sigma^{\ast}$, the sequence of singular values of $\WW^{\ast}$, satisfies $\boldsymbol\sigma^\ast \in \ell^{p,\infty}$, then there exist $\tilde\rho \in (0,1)$ and a constant $C>0$, depending on $\gamma_1^\epsilon, \gamma_2^\epsilon, \theta, \nu, \delta_0$ and $p$, such that with $\e_k \colonequals \tilde\rho^k$ and $s = 1/p-1/2$, we have
        \begin{align*}
            \norm{\WW_k-\WW^{\ast}}_{\frob} \leq C \e_k,\qquad \rank(\WW_k) \leq C^2 \norm{\boldsymbol\sigma^{\ast}}_{\ell^{p,\infty}}^{1/s} \e_k^{-1/s};
        \end{align*}
        \item if $\sigma^{\ast}_k \leq c_1 \exp(-c_2 k^\beta)$ for some constants $\beta,c_1,c_2>0$, then there exist $\tilde\rho\in(0,1)$ and a constant $C>0$, depending on $\gamma_1^\epsilon, \gamma_2^\epsilon, \theta, \nu, \delta_0, p, c_1, c_2$ and $\beta$, such that with $\e_k \colonequals \tilde \rho^k$, we have
        \begin{align*}
            \norm{\WW_k-\WW^{\ast}}_{\frob}\leq C \e_k,\qquad \rank(\bW_k)\leq C (1+|\ln(\e_k)|)^{1/\beta}.
        \end{align*}
    \end{enumerate}
\end{theorem}
Note that \Cref{alg:dynamical_thresholding} does not require the user to specify or estimate the decay behaviour of the singular values of the solution. Hence, \Cref{alg:dynamical_thresholding} automatically produces convergent iterates having quasi-optimal ranks.

\section{Numerical realization}
\label{sec:Numerical_realization}
In the following subsections we first describe two typical discretization schemes for the radiative transfer equation, focusing on the $S_N$-FEM method for a description of the computational complexity and memory requirements of key aspects of \Cref{alg:dynamical_thresholding}. We then present a procedure to lower such computational costs and storage requirements. 

\subsection{\texorpdfstring{$P_N$-FEM method}{PN-FEM method}}
\label{sec:PN-FEM_method}
Consider an equi-spaced partition $0=z_0<z_1<\ldots<z_J=Z$ of the interval $(0,Z)$, and let $\psi_j(z)$ be the piecewise linear functions on this partition such that $\psi_j(z_i)=\delta_{i,j}$.
For the discretization in $\mu$, denote by $L_n$ the Legendre polynomial of degree $n$, $n<N$ and $N$ odd, associated with the interval $(-1,1)$. We recall that $L_n$ is even if and only if $n$ is even. Since we only approximate the even part of the solution, we may therefore choose $H_n$ to be the restriction of $L_{2n}$ to $(0,1)$ such that $\int_0^1 H_n(\mu) H_m(\mu)\,d\mu=\delta_{m,n}$. Observe that with these choices, the spatial matrices have all size $(J+1)\times(J+1)$, being $J$ the number of spatial elements in the interval $(0,Z)$; the angular matrices have size $(M+1)\times (M+1)$, with $M:=(N-1)/2$. All matrices described in \Cref{sec:Structure_of_system_matrix} are sparse, except the boundary matrix $\MM_{\mu}(\mu)$; see, however, \cite{ES19} for a possible remedy.
Furthermore, $\MM_\mu = \II_\mu$. Since the Gram matrix $\DD + \MM_z$ of the $H^1(0,Z)$-inner product is not diagonal, we follow the strategy outlined in \Cref{rem:change_of_basis} and perform a change of basis using the Cholesky factorization $\DD + \MM_z = \TT_z \TT_z^T$. Since $\DD + \MM_z$ is tridiagonal, $\TT_z$ can be computed and applied in $\O(J)$ operations. Moreover, the inverse $\TT_z^{-1}$ can be applied in $\O(J)$ operations as well. Regarding the spectral bounds $\lambda$ and $\Lambda$ required in \Cref{sec:Preconditioning_via_exponential_sums_approximations}, we need to estimate the smallest and largest eigenvalues of $\MM_\mu(\mu^2)$ and $\TT_z^{-1}\MM_z \TT_z^{-T}$.
We have the following results:
\begin{lemma}
    There exists $c_N > 1$, with $|1-c_N| \sim N^{-2}$, such that
    \begin{equation*}
        c_N N^{-4} \norm{\vv}_2^2 \leq \langle \MM_\mu(\mu^2) \vv,\vv\rangle\leq \norm{\vv}_2^2 \quad \forall \vv\in\RRR^{M+1}.
    \end{equation*}
\end{lemma}
\begin{proof}
    Let $ v(\mu) = \sum_{m=0}^M v_m H_m(\mu)$. Then
    \begin{equation*}
        \langle \MM_\mu(\mu^2) \vv,\vv\rangle=\int_0^1 |v|^2\mu^2 \,d\mu \leq \int_0^1 |v|^2\,d\mu =\norm{\vv}_2^2,
    \end{equation*}
    showing the upper bound. Regarding the lower bound, let $\{(\mu_k,w_k)\}$ be a Gauss-Legendre quadrature rule with $2(M+1) = N+1$ points on $(0,1)$ and positive weights. Since the degree of exactness is $4M+3$ and $v^2\mu^2$ has a maximal degree of $4M+2$, it follows that
    \begin{equation*}
        \int_0^1 |v|^2\mu^2 \,d\mu=\sum_{k=0}^{2M+2} w_k v(\mu_k)^2\mu_k^2 \geq \min_{k}\mu_k^2 \int_0^1 v^2\,d\mu = \min_{k}\mu_k^2\norm{\vv}_2^2.
    \end{equation*}
    To complete the proof, we combine the following asymptotic formula for the integration points, see \cite{T50} or \cite{GG08},
    \begin{equation}
        \label{eq:zeros_Legendre_asymptotics}
        2\mu_k-1 = \left[ 1 - \dfrac{1}{8(N+1)^2} + \dfrac{1}{8(N+1)^3} \right]\cos\left( \pi \dfrac{4k-1}{4(N+1)+2} \right) + \mathcal O (N^{-4}),
    \end{equation}
    with Taylor expansion of the trigonometric function.
\end{proof}
\begin{lemma}
    \label{lem:bounds_H1-FEM} 
    There exists a constant $c > 0$ independent of $J$ such that
    \begin{equation*}
        cJ^{-2} \|\vv\|_2^2 \leq \langle \TT_z^{-1}\MM_z\TT_z^{-T} \vv,\vv\rangle\leq \|\vv\|_2^2\quad\forall \vv\in\RRR^{J+1}.
    \end{equation*}
\end{lemma}
\begin{proof}
    Since $\TT_z\TT_z^{T}=\DD+\MM_z$, the eigenvalues $\lambda$ of $ \TT_z^{-1}\MM_z\TT_z^{-T} $ satisfy
    \begin{equation*}
        (\DD+\MM_z) \vv = \lambda^{-1}\vv
    \end{equation*}
    with corresponding eigenvectors $\vv$. Standard inverse inequalities, see, e.g.,  \cite[p.~85]{LT03}, then allow to estimate $\lambda$ as asserted.
\end{proof}
\begin{remark}
    \label{rem:large_problems}
    The previous results imply that $\Lambda\leq 2$ and $\lambda \geq c_N N^{-2}+ cJ^{-2}$. Straightforward computation yields $c=12$. For very large problems, for instance $J = 10^{6}$, $N = 2^{15}+1$,  we obtain an exponential sum approximation with $\epsilon = 0.1$ with 17 terms, i.e. $i_1 = 3$ and $i_2 = 13$.
\end{remark}

\subsection{\texorpdfstring{$S_N$-FEM method}{SN-FEM method}}
\label{sec:SN-FEM_method}
We consider the same discretization of the $z$-dependence as in the $P_N$-FEM method.
In order to discretize the $\mu$-dependence, we let $\{H_n\}$ be piecewise constant functions associated to a partition $0=\mu_0<\mu_1<\ldots<\mu_N=1$ such that $H_n(\mu)=1/\sqrt{\mu_{n+1}-\mu_n}$ for $\mu\in(\mu_{n},\mu_{n+1})$ and $H_n(\mu)=0$ else.
For ease of discussion, we assume that this partition is equi-spaced with spacing $h_\mu=\mu_1-\mu_0$. 
We then obtain that all matrices are sparse, except $\SS$ which has rank one, and that $\MM_\mu=\II_\mu$. 
The spectral bounds $\lambda\geq c/J^2+ 1/(3N^2)$ and $\Lambda \leq 2 $ can then be computed from \Cref{lem:bounds_H1-FEM} and the next result.
\begin{lemma}
    Let $c_N=1/(3N^2)$. Then we have that
    \begin{equation*}
        c_N \|\vv\|_2^2 \leq \langle \MM_\mu(\mu^2) \vv,\vv\rangle\leq \|\vv\|_2^2\quad\forall \vv\in\RRR^{N}.
    \end{equation*}
\end{lemma}
\begin{proof}
    The claim follows from the observation that $\MM_\mu(\mu^2)$ is diagonal, with diagonal entries
    \begin{equation*}
        \int_0^1 \mu^2 H_n(\mu)^2\,d\mu=\frac{1}{h_\mu}\int_{\mu_{n}}^{\mu_{n+1}} \mu^2\,d\mu = \frac{1}{3}(3 \mu_n^2+ 3\mu_n h_\mu + h_\mu^2).
    \end{equation*}
\end{proof}
\begin{remark}
    A similar consideration as in \Cref{rem:large_problems} can be done also in this case for large problems. In general, for both $P_N$-FEM and $S_N$-FEM methods, moderate ranks of the preconditioner can be achieved. 
\end{remark}

\subsection{Computational complexity and storage requirements}
Let $r_p \colonequals i_1+i_2+1$ be the rank of the preconditioner $\PPP^{-1/2}$, which we recall to grow only logarithmically in the discretization parameters $J$ and $N$, see \Cref{sec:Preconditioning_via_exponential_sums_approximations}.
The computational complexity of producing the residual in line $4$ of \Cref{alg:dynamical_thresholding} is $\O(4r_p^2 JN)$ for the naive full tensor product representation of the iterate, and $\O(4r_P^2r (J+N))$ for the low-rank representation, where $r\leq \min\{J,N\}$ is the rank of the current iterate $\WW_k$, see \Cref{sec:rank_controlled_iteration}. 
The memory required to store the result of the application of $\AAA$ to $\WW_k$ is $\O(JN)$ for the full tensor product approach, and $O(4r_p^2r(J+N))$ for the low-rank approach. 
For rank truncation the intermediate rank $4 r_p^2 r$ will enter the computational cost cubically through the computation of a singular value decomposition.
Hence, large values of $4 r_p^2 r$ may render the method inefficient.
In order to further lower these computational costs and keep the ranks of the intermediate objects as low as possible, we next describe a modified version of \Cref{alg:dynamical_thresholding}. 


\subsection{Inexact evaluation of residuals}
\label{sec:Inexact_evaluation_residuals}

We consider \Cref{alg:inexact_residuals}, which computes the residual inexactly with the aim to keep the ranks of intermediate objects as small as possible, see \cite{BS17,BD16} and also \cite{GHS07}. 
For each iterate $\WW_k$, the residual $\RR_k$ is generated such that
\begin{equation*}
    \norm{\RR_k-(\AAA\WW_k-\FF)}_{\frob} < \eta_k, 
\end{equation*}
and the new candidate iterate $\WW_{k+1}$ is computed accordingly, see line 7. Here $\{\eta_k\}$ is a sequence of tolerances adjusted at each iteration to ensure that
\begin{equation*}
    \norm{\RR_k - (\AAA\WW_k-\FF)}_{\frob} \leq \min \{\tau_1\norm{\RR_k}_{\frob}, \tau_2\omega^{-1}\norm{\WW_{k+1}-\WW_k}_{\frob}\},
\end{equation*}
with user-defined parameters $\tau_1,\tau_2>0$. This is achieved by decreasing $\eta_k$ and recomputing $\RR_k$ (and the resulting new candidate iterate $\WW_{k+1}$) until
\begin{equation*}
    \eta_k \leq \min\{\tau_1\norm{\RR_k}_{\frob}, \tau_2\omega^{-1}\norm{\WW_{k+1}-\WW_k}_{\frob}\}.
\end{equation*}
If two consecutive iterates become too close, the thresholding parameter needs to be adapted, which is done in line $18$ if $\norm{\WW_{k+1}-\WW_k}_{\frob}\leq B\|\RR_{k+1}\|_{\frob}$, with 
\begin{align}\label{def:B}
    B\colonequals \dfrac{(1-\rho)(1-\tau_1)\nu}{(1+\tau_2)[\rho+(1-\tau_2)^{-1}(1+\rho)\tau_2]\gamma_2^{\epsilon}}.
\end{align}
This condition replaces the one in line $5$ of \Cref{alg:dynamical_thresholding} and ensures that 
\begin{equation}
    \label{eq:proof_sketch}
    \norm{\WW_{k+1}-\WW^\ast}_\frob \leq \dfrac{1}{1-\nu}\norm{\WW^\ast - \WW^{\delta_k}}_\frob,
\end{equation}
which allows to prove that $\WW_k \to \WW^\ast$, see \cite[Theorem 5.1]{BS17} (recall that $\WW^\ast$ and $\WW^{\delta_k}$ are the fixed points of the non-thresholded and the thresholded iterations, respectively). 
Furthermore, the additional check on $\eta_k$ in line $5$ ensures that if the cycle is exited, i.e. $\eta_k \leq C\norm{\RR_k}_\frob$, with 
\begin{align}\label{def:C}
    C\colonequals \min \left\{ \dfrac{(1-\tau_1)\tau_2B}{(1+\tau_1+\gamma_2^{\epsilon}B)\omega}, \dfrac{\rho\nu\tau_2(1-\tau_1)^2}{[\rho(1+\tau_1)(1+\tau_2) + \nu(1-\tau_1)(1-\rho)]\omega} \right\},
\end{align}
then the condition $\norm{\WW_{k+1}-\WW^{\delta_k}}_\frob \leq \norm{\WW_{k+1}-\WW^\ast}_\frob$ is satisfied, and therefore \cref{eq:proof_sketch} is guaranteed to hold, leading to the action on $\delta_k$. 
The statement of \Cref{thm:st_richardson}, with modified constants, applies also to \Cref{alg:inexact_residuals}, see \cite[Proposition 5.9]{BS17}. 

Next, we explain how one can obtain a sufficiently accurate approximation $\RR_k$ of $\AAA\WW_k-\FF$ that can be computed with less effort.
%
%
The computation of $\RR_k$ essentially reduces to an inexact evaluation of $\AAA\WW_k$.
Hence, we consider the routine $\WW_{\eta_k} \colonequals \mathtt{APPLY}(\WW_k, \eta_k)$ such that
\begin{equation}\label{eq:err_res_inexact}
    \norm{\AAA\WW_k - \WW_{\eta_k}}_{\frob} \leq \dfrac{\eta_k}{2}.
\end{equation}
The routine $\mathtt{APPLY}(\WW_k, \eta_k)$ is based on a careful summation procedure described in \cite[Section 7.2]{BD16}, which we adapt to our situation.
Recalling the structure \cref{eq:definition_P^(-1/2)} of $\PPP^{-1/2}$, we can write
\begin{equation*}
    \PPP^{-1/2} = \sum_{m=1}^{r_p} \Theta_m, \quad \Theta_m = \dfrac{h}{\sqrt{\pi}} \alpha_{f(m)}(1/2) \exp(-\rho_{f(m)} \MM_{\mu}(\mu^2)) \otimes \exp(-\rho_{f(m)}\MM_z),
\end{equation*}
with $f(m) = i_1 + \frac{i_2-i_1}{r_p-1}(m-1)$. With this notation,
\begin{equation}
    \label{eq:application_A_Wk}
    \AAA\WW_k = \sum_{m_0=1}^{r_p}\sum_{m_1=1}^{r_p} \Theta_{m_0}\EEE\Theta_{m_1}\WW_k.
\end{equation}    
The approximation procedure to retrieve $\WW_{\eta_k}$ is now divided in three parts:
\textit{(i)} Compute all $\WW_k^{m_0,m_1}:=\Theta_{m_0}\EEE\Theta_{m_1}\WW_k$, $1\leq m_0,m_1,\leq r_p$,
\textit{(ii)} discard certain $\WW_k^{m_0,m_1}$ according to some tolerance criterion, and \textit{(iii)} perform summation of kept $\WW_k^{m_0,m_1}$ with intermediate compression.
Let $\boldsymbol\tau\in\RRR^{r_p^2}$ be the vector with elements
\begin{equation*}
    \tau_q \colonequals \norm{\WW^{m_0,m_1}_k}_\frob, \quad m_0,m_1=1,\dots,r_p \quad\text{and}\quad q=r_p(m_1-1)+m_0,
\end{equation*}
and $\tilde{\boldsymbol{\tau}}$ the vector obtained from $\boldsymbol{\tau}$ by reordering its elements in a nondecreasing fashion. 
To determine the minor contributions to $\AAA\WW_k$, let $q_0$ be the largest integer such that 
\begin{equation*}
    \sum_{q=1}^{q_0} \tilde{\tau}_q \leq \dfrac{\eta_k}{2}.
\end{equation*}
Denote $m_0(q),\ m_1(q)$ the indices such that $\WW_{k}^{m_0(q),m_1(q)}$ corresponds to $\tilde\tau_{q}$.
Define $\WW_{\eta_k}^1=\WW_k^{m_0(q_0+1),m_0(q_0+1)}$, and for $1\leq q\leq r_p^2-1$, we define iteratively
\begin{align*}
    \WW_{\eta_k}^{q+1} := \texttt{TSVD}(\WW_{\eta_k}^q+\WW_{k}^{m_0(q_0+q),m_1(q_0+q)},\zeta_q),
\end{align*}
where $\texttt{TSVD}(\WW,\zeta)$ computes the truncated singular value decomposition of $\WW$ with Frobenius norm error less or equal to $\zeta$. Moreover, to guarantee that \eqref{eq:err_res_inexact} holds, $\zeta_q$ are chosen as follows
%
\begin{equation*}
    \zeta_q \colonequals \dfrac{\eta_k \sum_{s=q_0+1}^q \tilde{\tau}_s}{2\sum_{s=q_0+1}^{r_p^2}(r_p^2+1-s)\tilde{\tau}_s}.
\end{equation*}
We then set $\WW_{\eta_k}:=\WW_{\eta_k}^{r_p^2}$.
With this strategy, the ranks of the intermediate objects produced when computing the residuals in lines $6$ and $12$ of \Cref{alg:inexact_residuals} are kept small in practice, see \Cref{sec:Rank-1_manufactured_solution}, \Cref{tab:TC1_SN_FEM_IR} and \Cref{fig:TC1_intermediate_ranks}.
%
\begin{algorithm}
    \caption{ $\WW^\e={\rm STSolve2}(\AAAw,\FFw;\e)$ adapted from \cite[Algorithm~3]{BS17}.
    \label{alg:inexact_residuals}}
    \begin{algorithmic}[1]
        \REQUIRE $\delta_0\geq \omega \norm{\FF}_{\frob}$, $\nu,\theta, \eta_0, \tau_1\in (0,1)$, $\tau_2\in (0,\frac{1}{2}(1-\rho))$, $\gamma_1^\epsilon, \gamma_2^\epsilon, \omega, \rho$ as in \Cref{thm:convergence_standard_richardson}, $B$ as in \cref{def:B} and $C$ as in \cref{def:C}
        \STATE $\WW_0 \colonequals 0$, $\RR_0 \colonequals - \FF$, $k \colonequals 0$
        \WHILE{$ \norm{\RR_k}_{\frob} + \eta_k > \gamma_1^\epsilon \e$}
            \STATE $\WW_{k+1} \colonequals \S_{\delta_k}(\WW_k- \omega \RR_k)$
            \WHILE{$ \eta_k > \tau_2 \omega^{-1} \norm{\WW_{k+1}-\WW_k}_{\frob} $ and $\eta_k > C\norm{\RR_k}_{\frob}$}
                \STATE $\eta_k \leftarrow \frac{1}{2}\eta_k$
                \STATE compute $\RR_k$ such that $\norm{\RR_k-(\AAA\WW_k-\FF)}_{\frob} \leq \eta_k$ 
                \STATE $\WW_{k+1} \leftarrow \S_{\delta_k}(\WW_k- \omega \RR_k)$
            \ENDWHILE
            \STATE $\eta_{k+1} \colonequals 2\eta_k$
            \REPEAT
                \STATE $\eta_{k+1} \leftarrow \frac{1}{2}\eta_{k+1}$
                \STATE compute $\RR_{k+1}$ such that $\norm{\RR_{k+1}-(\AAA\WW_{k+1}-\FF)}_{\frob} \leq \eta_{k+1}$
                \IF{ $\norm{\RR_{k+1}}_{\frob}+\eta_{k+1} \leq \gamma_1^\epsilon \varepsilon $ }
                    \STATE set $\WW^\varepsilon \colonequals \WW_{k+1}$ and stop
                \ENDIF
            \UNTIL{ $\eta_{k+1} \leq \tau_1 \norm{\RR_{k+1}}_{\frob}$ }
            \IF{$\norm{\WW_{k+1}-\WW_k}_{\frob} \leq B \norm{\RR_{k+1}}_{\frob}$}
                \STATE $\delta_{k+1} \colonequals \theta\delta_k$
                \STATE $\eta_{k+1} \leftarrow \tau_1\norm{\RR_{k+1}}_{\frob}$
            \ELSE
                \STATE $\delta_{k+1} \colonequals \delta_k$
            \ENDIF
            \STATE $k \leftarrow k+1$
        \ENDWHILE
        \STATE $\WW^\varepsilon \colonequals \WW_k$
    \end{algorithmic}
\end{algorithm}

\begin{remark}
    From a practical implementation standpoint, all terms $\Theta_{m_0}\EEE\Theta_{m_1}$ must be computed to obtain the Frobenius norms $\tau_q$, though only a few need to be stored. Consequently, the computational complexity involved in producing $\WW_{\eta_k}$ is comparable to that of the straightforward $\AAA\WW_k$, but with greatly reduced memory requirements. The actual reduction in computational complexity occurs in the soft-thresholding step, because $\WW_{\eta_k}$ has a significantly smaller rank than $\AAA\WW_k$ in practice (and this rank enters cubically in the computational complexity of the soft-thresholding step).
\end{remark}

\section{Numerical experiments}
\label{sec:Numerical_experiments}
In this section we present numerical experiments to illustrate applications of \Cref{alg:dynamical_thresholding}. 
We show that, upon mesh refinement, the final iterates have approximation errors that behave like the error of the underlying Galerkin discretization, while the ranks of the corresponding iterates remain moderately small.  

\subsection{Rank-1 manufactured solution}
\label{sec:Rank-1_manufactured_solution}
We fix optical parameters $\sigma_a(z) = 3$ and $\sigma_s(z) = 1+\frac{1}{2}\sin(\pi z)$ and consider a rank-$1$ manufactured solution 
\begin{equation}
    \label{eq:test_case_1}
    \varphi_1(z, \mu) = |\mu|e^{-\mu}e^{-z(1-z)}, 
\end{equation}
with even part $u_1(z,\mu)=|\mu|\cosh(\mu) e^{-z(1-z)}$ having also rank one. 
For our experiments, we employ \Cref{alg:dynamical_thresholding} with initial thresholding parameter $\delta_0 = 10^{-1}$, reduction factor for the thresholding parameter $\theta = 0.75$, tolerance for the iterative method $\e = 10^{-7}$, and $\nu = 0.5$. 
We denote with $u_{J,N}$ the back transformation of the output of \Cref{alg:dynamical_thresholding}, through the preconditioner $\PPPw^{-1/2}$. 
To implement $\PPPw^{-1/2}$ we choose $\epsilon=1/10$. This leads to $\rho=0.9683$ as specified in \Cref{thm:convergence_standard_richardson}, and an expected number of iterations like $\log(\varepsilon)/\log(\rho)\approx 500$.

\cref{tab:TC1_PN_FEM} and \cref{tab:TC1_SN_FEM} show the behaviour of the $L^2$- and $W^2_G$-errors, together with the corresponding convergence rates upon mesh refinement, for the $P_N$-FEM and $S_N$-FEM methods respectively, as described in \Cref{sec:PN-FEM_method} and \Cref{sec:SN-FEM_method}. 
We observe that employing $N \sim J^{2/3}$ (motivated for instance in \cite{ES19}) for the $P_N$ method, and $N \sim J$ for the $S_N$ method, we retain the optimal convergence rate of $1$ for both error norms. 
We obtain ranks $r_p$ of $\PPPw^{-1/2}$ ranging from $8$ to $10$ and from $11$ to $13$ for the $P_N$-FEM and the $S_N$-FEM method for $J$ from $128$ to $1024$, respectively.
Moreover, the number of iterations and the ranks of $\WW^\e$ (third and second-to-last columns of \Cref{tab:TC1_PN_FEM} and \Cref{tab:TC1_SN_FEM}, respectively) are robust under mesh refinement. 
The ranks of $\UU^{\e} := \PPPw^{-1/2}\WW^{\e}$ are bounded by product of the ranks of $\PPPw^{-1/2}$ and $\WW^\e$. Therefore the rank of $\UU^\e$ is only slightly larger than the rank of $\WW^\e$, see also \cref{rem:choice_of_parameters_in_exp_sum}. Since the exact rank of $\UU^\e$ with respect to the energy norm cannot be computed, we give for comparison the rank of $\UU^\e$ computed in the Euclidean setting.
In \Cref{fig:TC1_plots} we show the behaviour of the ranks of each iterate (left column) and of the thresholding parameter right column) as the iterative method progresses. The mild increase in the ranks during the iteration verifies the low computational costs. In particular, there are no intermediate high ranks of the iterates. 
We note, however, that the ranks of the corresponding Galerkin solution, which we closely approximate with the choice $\varepsilon=10^{-7}$, and its transformation to the $\bw$-variables are unknown.

In \Cref{tab:TC1_SN_FEM_IR} we reproduce \Cref{tab:TC1_SN_FEM} using \Cref{alg:inexact_residuals} instead of \Cref{alg:dynamical_thresholding}. The additional parameters are chosen as $\tau_1 = \frac{1-\rho}{4(3-\rho)}$ and $\tau_2 = \frac{1-\rho}{4}$, with initial $\eta_0 = 0.1$. We observe that the errors for the back-transformed solution are similar to the errors shown in \Cref{tab:TC1_SN_FEM}.
The iteration count increases roughly but a factor of five, which might be explained by the fact that the contraction rate is to be expected even closer to one than the one of the ideal Richardson iteration.
We, however, observe that the ranks (and thus the memory requirements) of $\WW_{\eta_k}$, which are computed as in \Cref{sec:Inexact_evaluation_residuals}, (second-to-last column) remain much smaller than $4r_p^2\rank{(\WW_k)}$ (last column). Additionally, \Cref{fig:TC1_intermediate_ranks} shows the evolution of the ranks of $\WW_{\eta_k}$.

Motivated by the observation that quasi-optimal ranks might be achieved if the error tolerances are not too strict, see, e.g.,\cite[Lemma~2.7]{DES21}, we also consider the $S_N$-FEM method with variable tolerance of $\e=0.1/J$ in \Cref{alg:dynamical_thresholding}.
As shown in \Cref{tab:TC1_SN_FEM_adaptive_tolerance}, this choice lowers both the number of iterations (which is no longer constant, but depends on the mesh) and the final rank of $\WW^{\e}$, which remains close to the optimal rank, speeding up convergence but at the same time retaining the optimal convergence rate for both norms. 

\begin{table}[ht!]
    \label{tab:TC1_PN_FEM}
    \caption{Discretization $L^2$- and $W^2_G$-errors, with convergence rates, and number of iterations ($N_{\mathrm{it}}$) of \Cref{alg:dynamical_thresholding} for the $P_N$-FEM method, with $J$ elements in space and polynomials of maximum degree $N$ in angle, applied to test case \cref{eq:test_case_1}. Last two columns: final rank of the output of the algorithm and of its back transformation through the preconditioner.}
    \centering\small\setlength\tabcolsep{0.75em}
    \begin{tabular}{c c c c c c c | c c} 
	   \toprule
	   \multicolumn{9}{c}{$P_N$-FEM method, $\e = 10^{-7}$}\\
	   \cmidrule{1-9}
        $J$ & $N$ & $N_{\mathrm{it}}$ & $\norm{u_1-u_{J,N}}_{L^2(\Omega)}$ & rate & $\norm{u_1-u_{J,N}}_{W^2_G(\Omega)}$ & rate & $\mathtt{r}(\WW^{\e})$ & $\mathtt{r}(\UU^\e)$ \\
        \midrule
        128  & 27  & 286 & $2.72\cdot10^{-3}$ & 0.98 & $4.01\cdot10^{-3}$ & 0.99 & 10 & 13 \\
        256  & 41  & 274 & $1.47\cdot10^{-3}$ & 0.89 & $2.08\cdot10^{-3}$ & 0.95 & 10 & 15 \\
        512  & 65  & 267 & $7.41\cdot10^{-4}$ & 0.99 & $1.04\cdot10^{-3}$ & 1.00 & 11 & 16 \\
        1024 & 103 & 256 & $3.73\cdot10^{-4}$ & 0.99 & $5.21\cdot10^{-4}$ & 0.99 & 11 & 17 \\
        2048 & 163 & 257 & $1.88\cdot10^{-4}$ & 0.99 & $2.61\cdot10^{-4}$ & 1.00 & 12 & 19 \\
        4096 & 257 & 247 & $9.50\cdot10^{-5}$ & 0.98 & $1.31\cdot10^{-4}$ & 0.99 & 12 & 20 \\
        \bottomrule
    \end{tabular}
\end{table}

\begin{table}[ht!]
    \label{tab:TC1_SN_FEM}
    \caption{Discretization $L^2-$ and $W^2_G-$errors, with convergence rates, and number of iterations ($N_{\mathrm{it}}$) of \Cref{alg:dynamical_thresholding} for the $S_N$-FEM method, with $J$ elements in space and $N$ subdivisions in angle, applied to test case \cref{eq:test_case_1}. Last two columns: final rank of the output of the algorithm and of its back transformation through the preconditioner.}
    \centering\small\setlength\tabcolsep{0.75em}
    \begin{tabular}{c c c c c c c | c c} 
	   \toprule
	   \multicolumn{9}{c}{$S_N$-FEM method, $\e = 10^{-7}$}\\
	   \cmidrule{1-9}
        $J$ & $N$ & $N_{\mathrm{it}}$ & $\norm{u_1-u_{J,N}}_{L^2(\Omega)}$ & rate & $\norm{u_1-u_{J,N}}_{W_G^2(\Omega)}$ & rate & $\mathtt{r}(\WW^{\e})$ & $\mathtt{r}(\UU^\e)$ \\
        \midrule
        128  & 128  & 232 & $3.12\cdot10^{-3}$ & 1.00 & $4.45\cdot10^{-3}$ & 1.00 & 16 & 24 \\
        256  & 256  & 223 & $1.56\cdot10^{-3}$ & 1.00 & $2.23\cdot10^{-3}$ & 1.00 & 16 & 25 \\
        512  & 512  & 216 & $7.79\cdot10^{-4}$ & 1.00 & $1.11\cdot10^{-3}$ & 1.00 & 15 & 27 \\
        1024 & 1024 & 215 & $3.81\cdot10^{-4}$ & 1.00 & $5.56\cdot10^{-4}$ & 1.00 & 16 & 28 \\
        \bottomrule
    \end{tabular}
\end{table}

\begin{table}[ht!]
    \label{tab:TC1_SN_FEM_IR}
    \caption{Discretization $L^2-$ and $W^2_G-$errors, with convergence rates, and number of iterations ($N_{\mathrm{it}}$) of \Cref{alg:inexact_residuals} for the $S_N$-FEM method, with $J$ elements in space and $N$ subdivisions in angle, applied to test case \cref{eq:test_case_1}. Last four columns: final rank of the output of the algorithm and of its back transformation through the preconditioner, and maximum rank of the computed intermediate object, i.e. $\mathtt{r}_{\mathrm{inexact}} \colonequals \max\{\rank(\WW_{\eta_k}),\,k=1,\dots,N_{\mathrm{it}})\}$, compared with the naive rank we would obtain without the summation procedure, i.e. $\mathtt{r}_{\mathrm{naive}} = \max\{4r_p^2 \rank(\WW_k),\,k=1,\dots,N_{\mathrm{it}}\}$.}
    \centering\small\setlength\tabcolsep{0.75em}
    \begin{tabular}{c c c c c | c c c c} 
	   \toprule
	   \multicolumn{9}{c}{$S_N$-FEM method, $\e = 10^{-7}$}\\
	   \cmidrule{1-9}
        $J$ & $N$ & $N_{\mathrm{it}}$ & $\norm{u_1-u_{J,N}}_{L^2(\Omega)}$ & $\norm{u_1-u_{J,N}}_{W_G^2(\Omega)}$ & $\mathtt{r}(\WW^{\e})$ & $\mathtt{r}(\UU^\e)$ & $\mathtt{r}_{\mathrm{inexact}}$ & $\mathtt{r}_{\mathrm{naive}}$\\
        \midrule
        128  & 128  & 1499 & $3.12\cdot10^{-3}$ & $4.45\cdot10^{-3}$ & 15 & 23 & 121 & 7260 \\
        256  & 256  & 1415 & $1.56\cdot10^{-3}$ & $2.23\cdot10^{-3}$ & 16 & 25 & 144 & 9216 \\
        512  & 512  & 1507 & $7.79\cdot10^{-4}$ & $1.11\cdot10^{-3}$ & 17 & 27 & 144 & 9792 \\
        1024 & 1024 & 1396 & $3.90\cdot10^{-4}$ & $5.56\cdot10^{-4}$ & 17 & 28 & 169 & 11490 \\
        \bottomrule
    \end{tabular}
\end{table}

\begin{table}[ht!]
    \label{tab:TC1_SN_FEM_adaptive_tolerance}
    \caption{Discretization $L^2-$ and $W^2_G-$errors, with convergence rates, and number of iterations ($N_{\mathrm{it}}$) of \Cref{alg:dynamical_thresholding} for the $S_N$-FEM method, with $J$ elements in space and $N$ subdivisions in angle, applied to test case \cref{eq:test_case_1} with Richardson tolerance chosen as $0.1/J$. Last two columns: final rank of the output of the algorithm and of its back transformation through the preconditioner.}
    \centering\small\setlength\tabcolsep{0.75em}
    \begin{tabular}{c c c c c c c | c c} 
	   \toprule
	   \multicolumn{9}{c}{$S_N$-FEM method, $\e = 0.1/J$}\\
	   \cmidrule{1-9}
        $J$ & $N$ & $N_{\mathrm{it}}$ & $\norm{u_1-u_{J,N}}_{L^2(\Omega)}$ & rate & $\norm{u_1-u_{J,N}}_{W_G^2(\Omega)}$ & rate & $\mathtt{r}(\WW^{\e})$ & $\mathtt{r}(\UU^\e)$ \\
        \midrule
        128  & 128  & 88 & $3.12\cdot10^{-3}$ & 1.00 & $4.46\cdot10^{-3}$ & 1.00 & 3 & 18 \\
        256  & 256  & 98 & $1.56\cdot10^{-3}$ & 1.00 & $2.23\cdot10^{-3}$ & 1.00 & 4 & 20 \\
        512  & 512  & 107 & $7.79\cdot10^{-4}$ & 1.00 & $1.11\cdot10^{-3}$ & 1.00 & 4 & 21 \\
        1024 & 1024 & 116 & $3.90\cdot10^{-4}$ & 1.00 & $5.57\cdot10^{-4}$ & 1.00 & 4 & 21 \\
        \bottomrule
    \end{tabular}
\end{table}

\begin{figure}[ht!]
    \centering
    \label{fig:TC1_plots}
	\includegraphics[width=.49\textwidth]{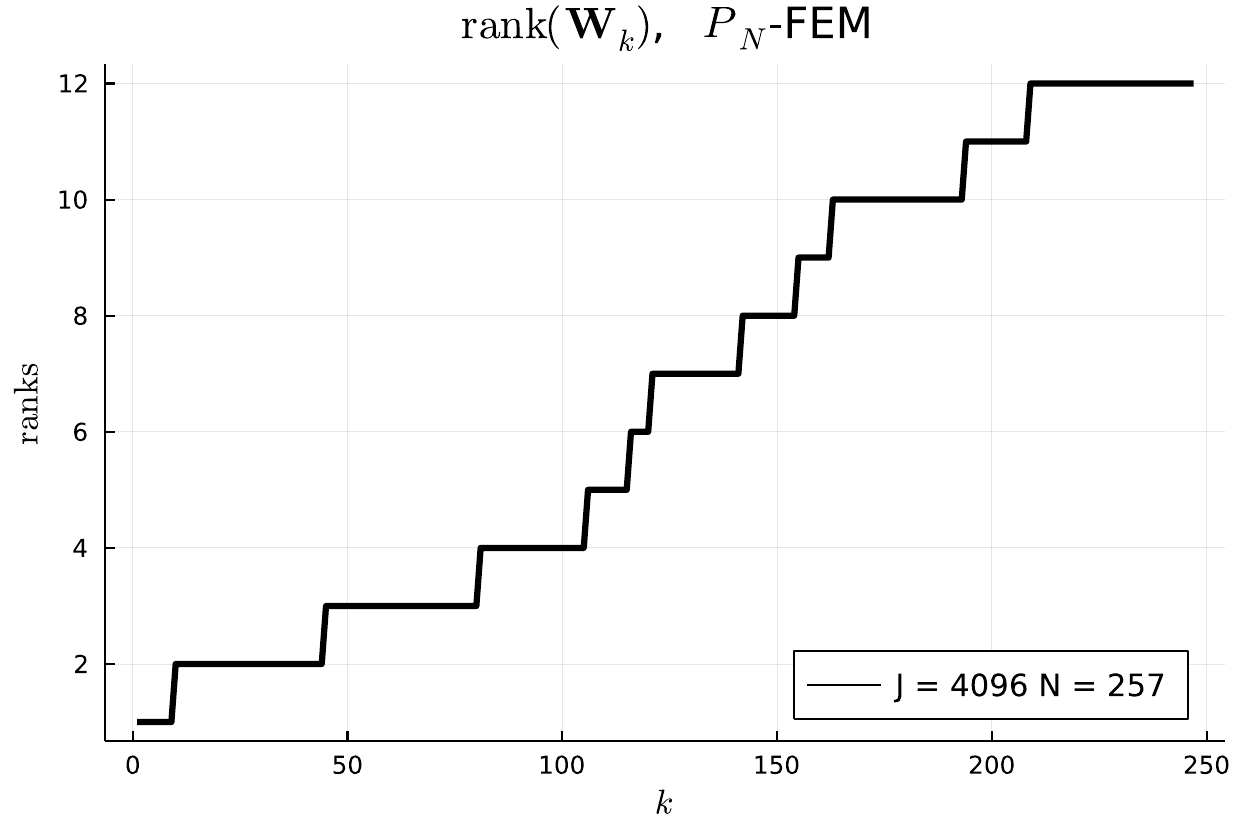}
	\includegraphics[width=.49\textwidth]{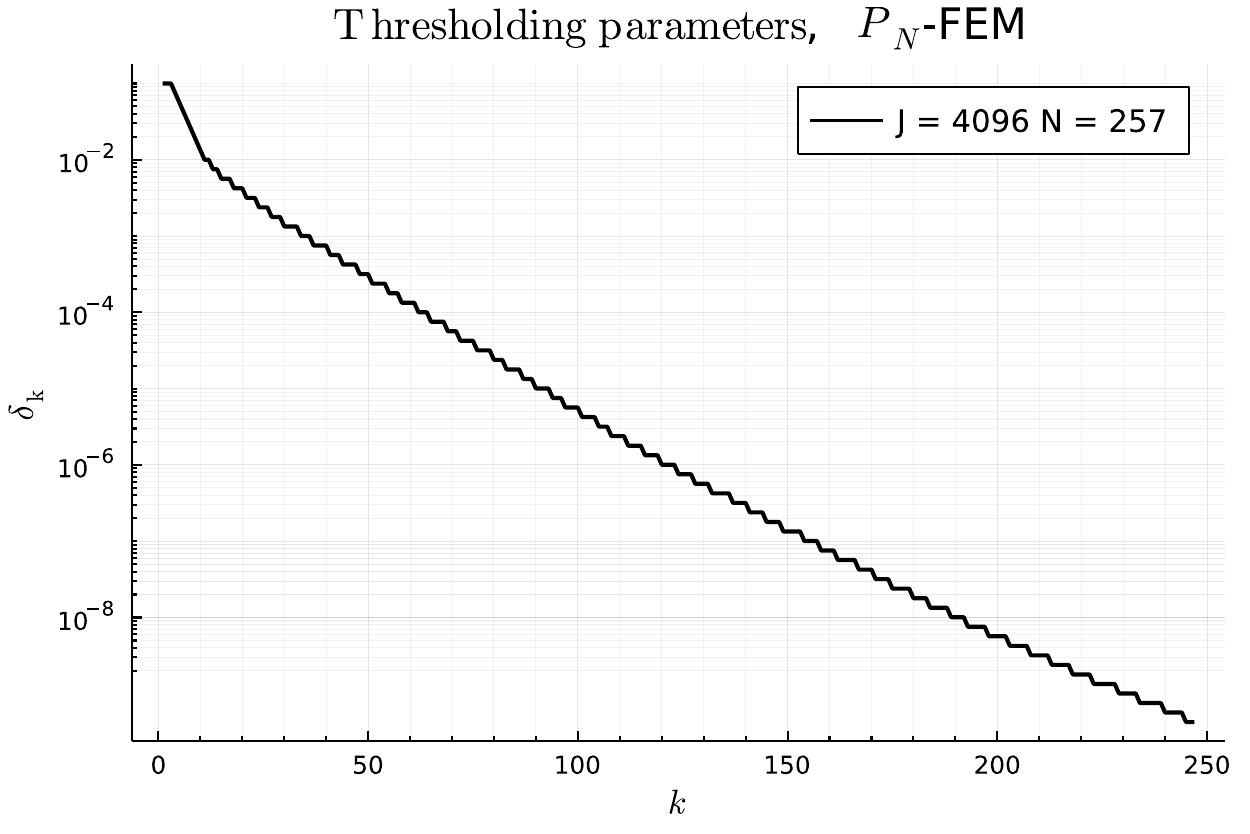}\\
	\includegraphics[width=.49\textwidth]{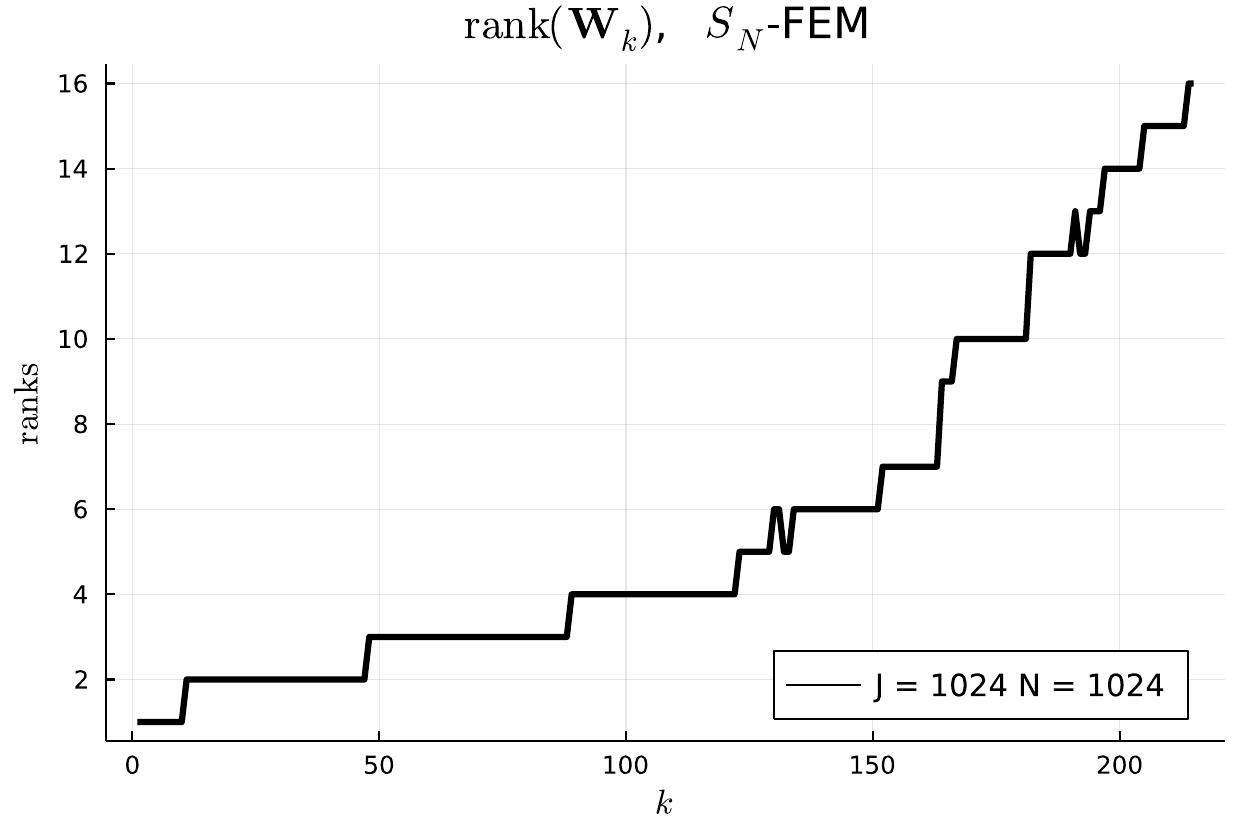}
    \includegraphics[width=.49\textwidth]{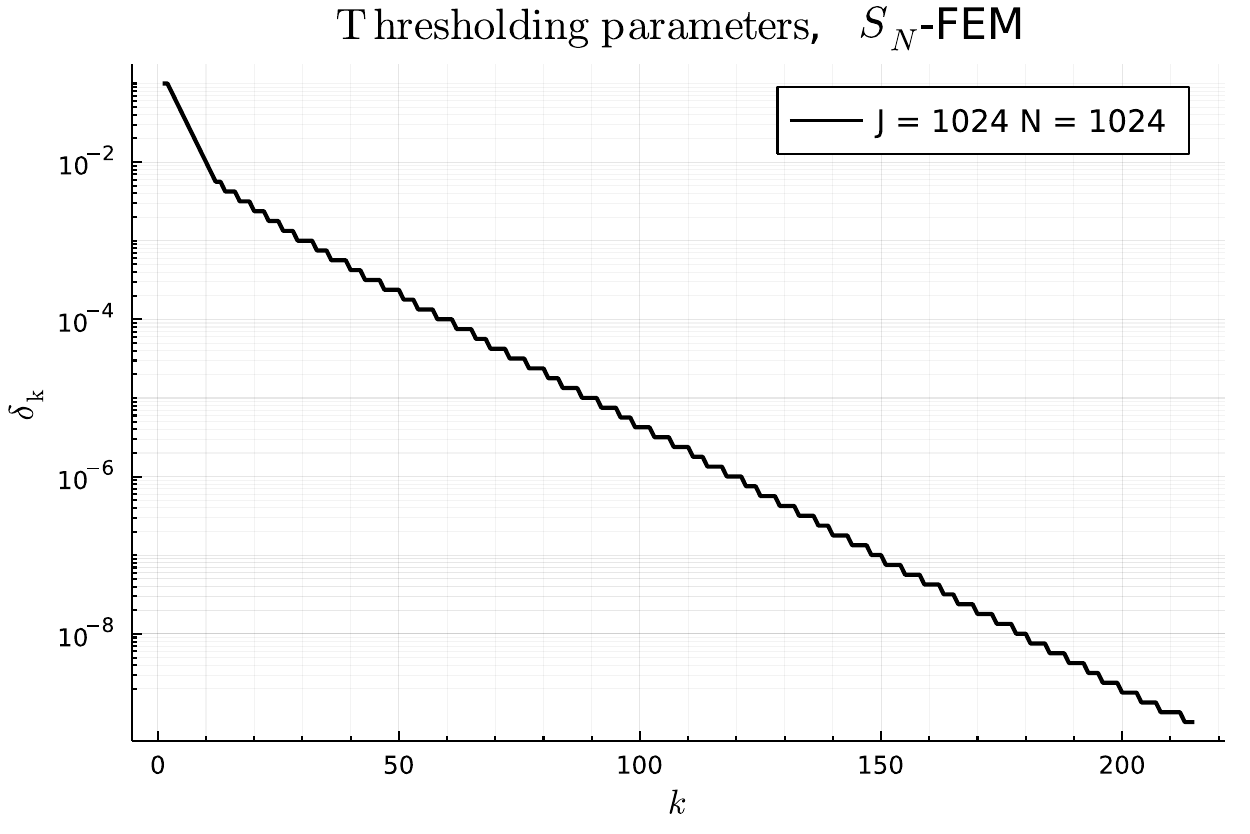}
	\caption{Behaviour of iterates ranks (left) and thresholding parmeter (right) for the $P_N$-FEM (top) and $S_N$-FEM (bottom) methods applied to test case \cref{eq:test_case_1} with Richardson tolerance $\e=10^{-7}$.}
\end{figure}

\begin{figure}[ht!]
    \centering
    \label{fig:TC1_intermediate_ranks}
	\includegraphics[width=.69\textwidth]{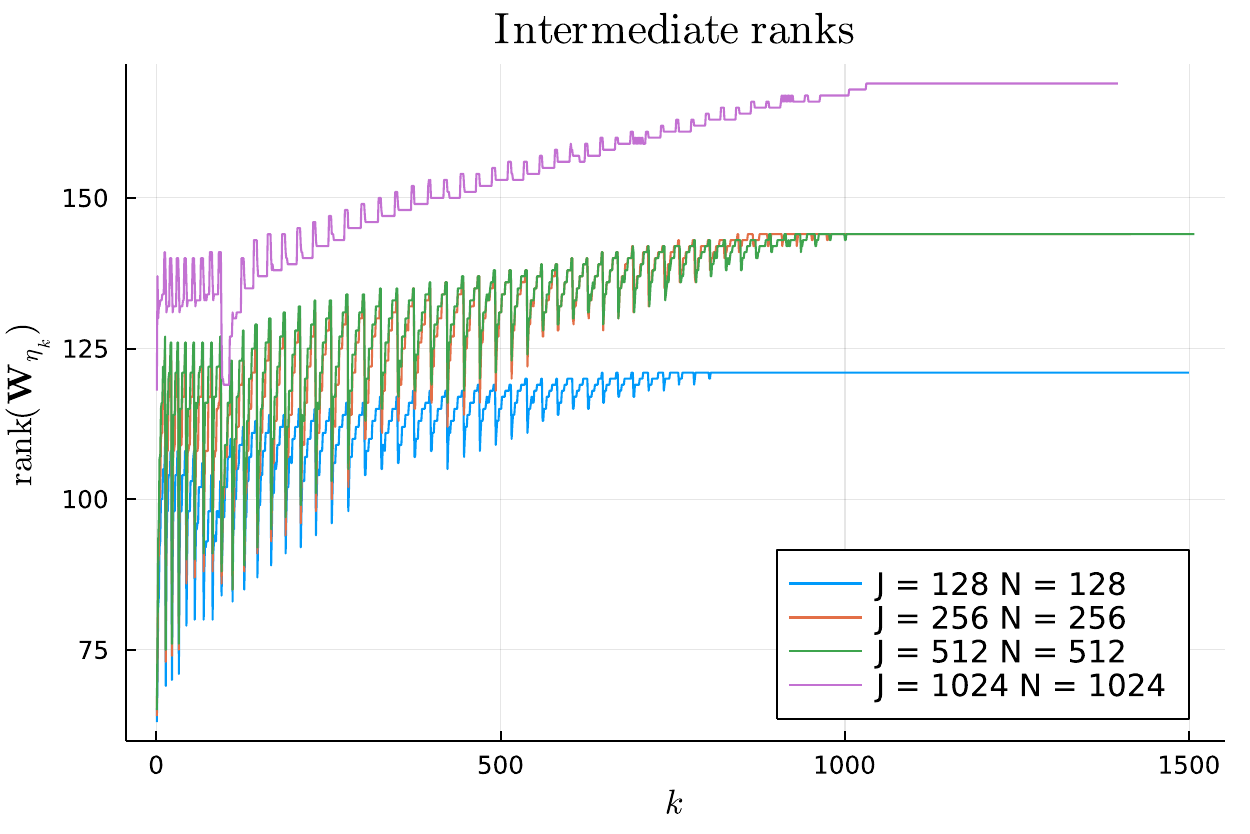}
	\caption{Ranks of the intermediate objects $\WW_{\eta_k}$ produced in lines $5$ and $12$ of \Cref{alg:inexact_residuals} for the $S_N$-FEM method, with $J$ elements in space and $N$ subdivisions in angle, applied to test case \cref{eq:test_case_1}.}
\end{figure}

\subsection{Exact solution of infinite rank}
\label{sec:Exact_solution_of_infinite_rank}
We consider now a manufactured solution with (potentially) infinite rank. The optical parameters are the same as in \Cref{sec:Rank-1_manufactured_solution}, as well as the setup of the iterative method. The exact solution is built through its $L^2(\Omega)$-SVD as follows
\begin{equation}
    \label{eq:test_case_2}
    \varphi(z, \mu) = u_2(z, \mu) = 2\sum_{k=1}^{\infty}\sigma_k\sin(k\pi z)\cos(k\pi\mu).
\end{equation}
For the experiments below, we choose $\sigma_k = k^{-3}$ for the algebraic behaviour (it can be proved that $\boldsymbol\sigma \in \ell^{2/7, \infty} \subset \ell^2$ in this case), and $\sigma_k = \exp(-k^2)$ for the exponential behaviour. 
We observe that, as in the previous subsection, the mentioned decay rates for the singular values hold for $u_2$ in the $L^2(\Omega)$-norm, but not necessarily in the graph norm. Moreover, we do not know either the behavior of the singular values of the corresponding transformation of $u_2$ under $\PPPw^{1/2}$ or the corresponding behavior for the Galerkin approximation.

In \Cref{tab:TC2_PN_FEM_algebraic,tab:TC2_PN_FEM_exponential} we report the results for the $P_N$-FEM method. We observe optimal convergence rates for the error when $J$ is increased. At the same time, the iteration count with fixed tolerance $\varepsilon=10^{-7}$ as well as the resulting ranks are robust upon mesh refinement.

\begin{table}[ht!]
    \label{tab:TC2_PN_FEM_algebraic}
    \caption{Discretization $L^2$- and $W^2_G$-errors, with convergence rates, and number of iterations ($N_{\mathrm{it}}$) of \Cref{alg:dynamical_thresholding} for the $P_N$-FEM method, with $J$ elements in space and polynomials of maximum degree $N$ in angle, applied to test case \cref{eq:test_case_2} with $\sigma_k = k^{-3}$. Last two columns: final rank of the output of the algorithm and of its back transformation through the preconditioner.}
    \centering\small\setlength\tabcolsep{0.75em}
    \begin{tabular}{c c c c c c c | c c} 
	   \toprule
	   \multicolumn{9}{c}{$P_N$-FEM method}\\
	   \cmidrule{1-9}
        $J$ & $N$ & $N_{\mathrm{it}}$ & $\norm{u_2-u_{J,N}}_{L^2(\Omega)}$ & rate & $\norm{u_2-u_{J,N}}_{W_G^2(\Omega)}$ & rate & $\mathtt{r}(\WW^{\e})$ & $\mathtt{r}(\UU^\e)$ \\
        \midrule
        128  & 27  & 435 & $1.62\cdot10^{-3}$ & 1.61 & $2.54\cdot10^{-2}$ & 1.10 & 14 & 14 \\
        256  & 41  & 447 & $5.52\cdot10^{-4}$ & 1.55 & $1.12\cdot10^{-2}$ & 1.18 & 21 & 21 \\
        512  & 65  & 444 & $2.47\cdot10^{-5}$ & 4.48 & $4.27\cdot10^{-3}$ & 1.39 & 27 & 30 \\
        1024 & 103 & 445 & $8.76\cdot10^{-7}$ & 4.82 & $2.13\cdot10^{-3}$ & 1.00 & 28 & 32 \\
        2048 & 163 & 445 & $4.87\cdot10^{-7}$ & 0.85 & $1.07\cdot10^{-3}$ & 1.00 & 30 & 34 \\
        4096 & 257 & 445 & $5.92\cdot10^{-8}$ & 3.04 & $5.33\cdot10^{-4}$ & 1.00 & 31 & 36 \\
        \bottomrule
    \end{tabular}
\end{table}

\begin{table}[ht!]
    \label{tab:TC2_PN_FEM_exponential}
    \caption{Discretization $L^2$- and $W^2_G$-errors, with convergence rates, and number of iterations ($N_{\mathrm{it}}$) of \Cref{alg:dynamical_thresholding} for the $P_N$-FEM method, with $J$ elements in space and polynomials of maximum degree $N$ in angle, applied to test case \cref{eq:test_case_2} with $\sigma_k = \exp(-k^2)$. Last two columns: final rank of the output of the algorithm and of its back transformation through the preconditioner.}
    \centering\small\setlength\tabcolsep{0.75em}
    \begin{tabular}{c c c c c c c | c c} 
	   \toprule
	   \multicolumn{9}{c}{$P_N$-FEM method, $\e = 10^{-7}$}\\
	   \cmidrule{1-9}
        $J$ & $N$ & $N_{\mathrm{it}}$ & $\norm{u_2-u_{J,N}}_{L^2(\Omega)}$ & rate & $\norm{u_2-u_{J,N}}_{W_G^2(\Omega)}$ & rate & $\mathtt{r}(\WW^{\e})$ & $\mathtt{r}(\UU^\e)$ \\
        \midrule
        128  & 27  & 337 & $1.70\cdot10^{-5}$ & 2.00 & $5.16\cdot10^{-3}$ & 1.00 & 11 & 14 \\
        256  & 41  & 335 & $4.25\cdot10^{-6}$ & 2.00 & $2.58\cdot10^{-3}$ & 1.00 & 13 & 17 \\
        512  & 65  & 335 & $1.07\cdot10^{-6}$ & 1.99 & $1.29\cdot10^{-3}$ & 1.00 & 13 & 19 \\
        1024 & 103 & 334 & $2.74\cdot10^{-7}$ & 1.97 & $6.46\cdot10^{-4}$ & 1.00 & 15 & 21 \\
        2048 & 163 & 333 & $9.42\cdot10^{-8}$ & 1.54 & $3.23\cdot10^{-4}$ & 1.00 & 16 & 23 \\
        4096 & 257 & 334 & $6.75\cdot10^{-8}$ & 0.48 & $1.61\cdot10^{-4}$ & 1.00 & 18 & 26 \\
        \bottomrule
    \end{tabular}
\end{table}

The corresponding results for the $S_N$-FEM method are reported in \Cref{tab:TC2_SN_FEM_algebraic} and \Cref{tab:TC2_SN_FEM_exponential}. Similar comments as for the $P_N$-FEM method apply.

\begin{table}[ht!]
    \label{tab:TC2_SN_FEM_algebraic}
    \caption{Discretization $L^2-$ and $W^2_G-$errors, with convergence rates, and number of iterations ($N_{\mathrm{it}}$) of \Cref{alg:dynamical_thresholding} for the $S_N$-FEM method, with $J$ elements in space and $N$ subdivisions in angle, applied to test case \cref{eq:test_case_2} with $\sigma_k = k^{-3}$. Last two columns: final rank of the output of the algorithm and of its back transformation through the preconditioner.}
    \centering\small\setlength\tabcolsep{0.75em}
    \begin{tabular}{c c c c c c c | c c} 
	   \toprule
	   \multicolumn{9}{c}{$S_N$-FEM method, $\e = 10^{-7}$}\\
	   \cmidrule{1-9}
        $J$ & $N$ & $N_{\mathrm{it}}$ & $\norm{u_2-u_{J,N}}_{L^2(\Omega)}$ & rate & $\norm{u_2-u_{J,N}}_{W_G^2(\Omega)}$ & rate & $\mathtt{r}(\WW^{\e})$ & $\mathtt{r}(\UU^\e)$ \\
        \midrule
        128  & 128  & 429 & $7.37\cdot10^{-3}$ & 1.00 & $2.41\cdot10^{-2}$ & 1.00 & 31 & 36 \\
        256  & 256  & 439 & $3.69\cdot10^{-3}$ & 1.00 & $1.21\cdot10^{-2}$ & 1.00 & 33 & 39 \\
        512  & 512  & 429 & $1.84\cdot10^{-3}$ & 1.00 & $6.03\cdot10^{-3}$ & 1.00 & 34 & 42 \\
        1024 & 1024 & 429 & $9.21\cdot10^{-4}$ & 1.00 & $3.01\cdot10^{-3}$ & 1.00 & 36 & 45 \\
        \bottomrule
    \end{tabular}
\end{table}

For the $S_N$-method, we report in \Cref{tab:TC2_SN_FEM_algebraic_adaptive_tolerance} and \Cref{tab:TC2_SN_FEM_exponential_adaptive_tolerance} the results for an iterative scheme with tolerance dependent on the mesh, i.e. $\e = 0.1/J$; see also the previous subsection for a motivation. These results show that the same optimal convergence rates can be achieved without being too greedy on the tolerance, keeping at the same time the ranks of the iterates smaller than the cases with $\e = 10^{-7}$.
Moreover, the increase in the ranks upon mesh refinement is small.

\begin{table}[ht!]
    \label{tab:TC2_SN_FEM_algebraic_adaptive_tolerance}
    \caption{Discretization $L^2-$ and $W^2_G-$errors, with convergence rates, and number of iterations ($N_{\mathrm{it}}$) of \Cref{alg:dynamical_thresholding} for the $S_N$-FEM method, with $J$ elements in space and $N$ subdivisions in angle, applied to test case \cref{eq:test_case_2} with $\sigma_k = k^{-3}$ and Richardson tolerance adaptively chosen as $0.1/J$. Last two columns: final rank of the output of the algorithm and of its back transformation through the preconditioner.}
    \centering\small\setlength\tabcolsep{0.75em}
    \begin{tabular}{c c c c c c c | c c} 
	   \toprule
	   \multicolumn{9}{c}{$S_N$-FEM method, $\e = 0.1/J$}\\
	   \cmidrule{1-9}
        $J$ & $N$ & $N_{\mathrm{it}}$ & $\norm{u_2-u_{J,N}}_{L^2(\Omega)}$ & rate & $\norm{u_2-u_{J,N}}_{W_G^2(\Omega)}$ & rate & $\mathtt{r}(\WW^{\e})$ & $\mathtt{r}(\UU^\e)$ \\
        \midrule
        128  & 128  & 208 & $7.37\cdot10^{-3}$ & 1.00 & $2.41\cdot10^{-2}$ & 1.00 & 23 & 35 \\
        256  & 256  & 226 & $3.69\cdot10^{-3}$ & 1.00 & $1.21\cdot10^{-2}$ & 1.00 & 25 & 38 \\
        512  & 512  & 242 & $1.84\cdot10^{-3}$ & 1.00 & $6.03\cdot10^{-3}$ & 1.00 & 26 & 40 \\
        1024 & 1024 & 260 & $9.21\cdot10^{-4}$ & 1.00 & $3.01\cdot10^{-3}$ & 1.00 & 27 & 43 \\
        \bottomrule
    \end{tabular}
\end{table}

\begin{table}[ht!]
    \label{tab:TC2_SN_FEM_exponential}
    \caption{Discretization $L^2-$ and $W^2_G-$errors, with convergence rates, and number of iterations ($N_{\mathrm{it}}$) of \Cref{alg:dynamical_thresholding} for the $S_N$-FEM method, with $J$ elements in space and $N$ subdivisions in angle, applied to test case \cref{eq:test_case_2} with $\sigma_k = \exp(-k^2)$. Last two columns: final rank of the output of the algorithm and of its back transformation through the preconditioner.}
    \centering\small\setlength\tabcolsep{0.75em}
    \begin{tabular}{c c c c c c c | c c} 
	   \toprule
	   \multicolumn{9}{c}{$S_N$-FEM method, $\e = 10^{-7}$}\\
	   \cmidrule{1-9}
        $J$ & $N$ & $N_{\mathrm{it}}$ & $\norm{u_2-u_{J,N}}_{L^2(\Omega)}$ & rate & $\norm{u_2-u_{J,N}}_{W_G^2(\Omega)}$ & rate & $\mathtt{r}(\WW^{\e})$ & $\mathtt{r}(\UU^\e)$ \\
        \midrule
        128  & 128  & 297 & $2.62\cdot10^{-3}$ & 1.00 & $7.30\cdot10^{-3}$ & 1.00 & 17 & 26 \\
        256  & 256  & 297 & $1.31\cdot10^{-3}$ & 1.00 & $3.65\cdot10^{-3}$ & 1.00 & 19 & 29 \\
        512  & 512  & 297 & $6.55\cdot10^{-4}$ & 1.00 & $1.83\cdot10^{-3}$ & 1.00 & 20 & 32 \\
        1024 & 1024 & 298 & $3.27\cdot10^{-4}$ & 1.00 & $9.13\cdot10^{-4}$ & 1.00 & 22 & 35 \\
        \bottomrule
    \end{tabular}
\end{table}

\begin{table}[ht!]
    \label{tab:TC2_SN_FEM_exponential_adaptive_tolerance}
    \caption{Discretization $L^2-$ and $W^2_G-$errors, with convergence rates, and number of iterations ($N_{\mathrm{it}}$) of \Cref{alg:dynamical_thresholding} for the $S_N$-FEM method, with $J$ elements in space and $N$ subdivisions in angle, applied to test case \cref{eq:test_case_2} with $\sigma_k = \exp(-k^2)$ and Richardson tolerance adaptively chosen as $0.1/J$. Last two columns: final rank of the output of the algorithm and of its back transformation through the preconditioner.}
    \centering\small\setlength\tabcolsep{0.75em}
    \begin{tabular}{c c c c c c c | c c} 
	   \toprule
	   \multicolumn{9}{c}{$S_N$-FEM method, $\e = 0.1/J$}\\
	   \cmidrule{1-9}
        $J$ & $N$ & $N_{\mathrm{it}}$ & $\norm{u_2-u_{J,N}}_{L^2(\Omega)}$ & rate & $\norm{u_2-u_{J,N}}_{W_G^2(\Omega)}$ & rate & $\mathtt{r}(\WW^{\e})$ & $\mathtt{r}(\UU^\e)$ \\
        \midrule
        128  & 128  & 127 & $2.62\cdot10^{-3}$ & 1.00 & $7.30\cdot10^{-3}$ & 1.00 & 9 & 25 \\
        256  & 256  & 138 & $1.31\cdot10^{-3}$ & 1.00 & $3.65\cdot10^{-3}$ & 1.00 & 9 & 28 \\
        512  & 512  & 150 & $6.55\cdot10^{-4}$ & 1.00 & $1.83\cdot10^{-3}$ & 1.00 & 10 & 31 \\
        1024 & 1024 & 162 & $3.27\cdot10^{-4}$ & 1.00 & $9.13\cdot10^{-4}$ & 1.00 & 12 & 33 \\
        \bottomrule
    \end{tabular}
\end{table}

\begin{figure}[ht!]
    \centering
    \label{fig:TC2_plots}
	\includegraphics[width=.49\textwidth]{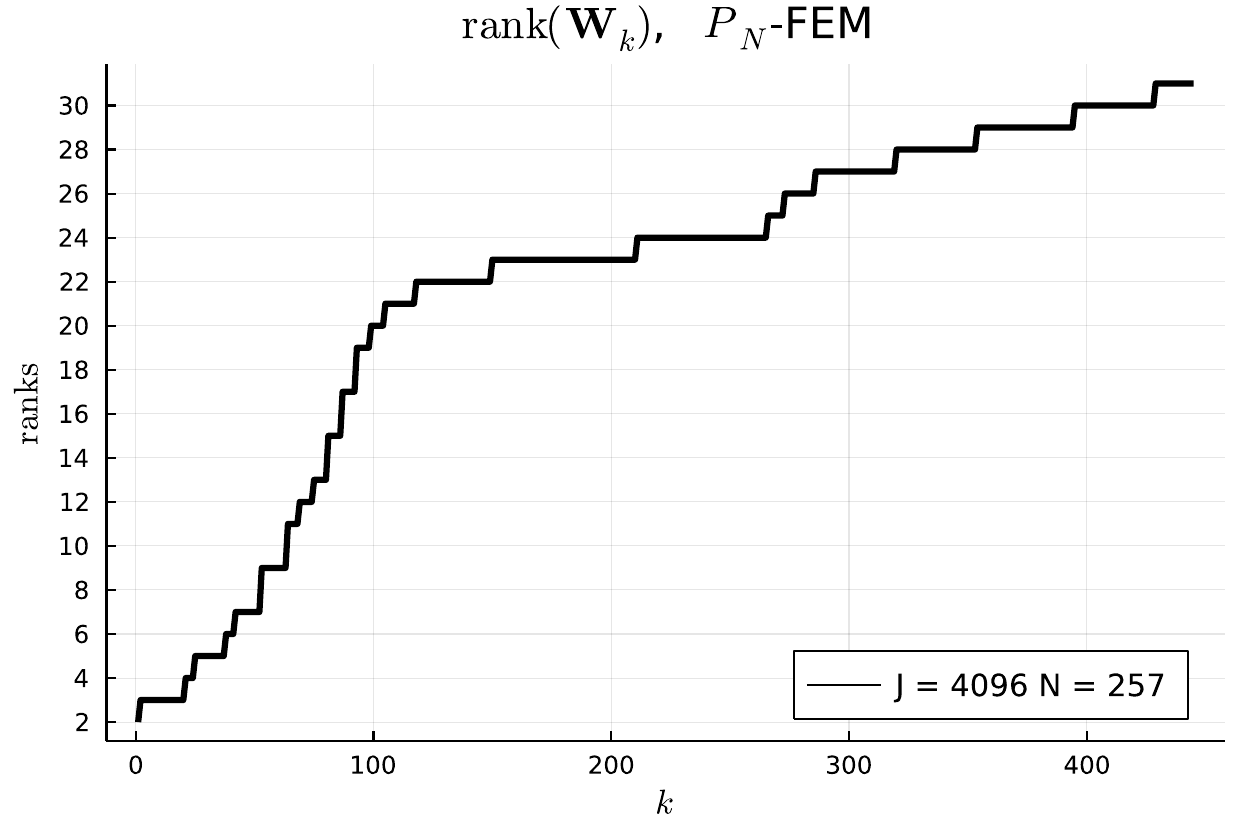}
	\includegraphics[width=.49\textwidth]{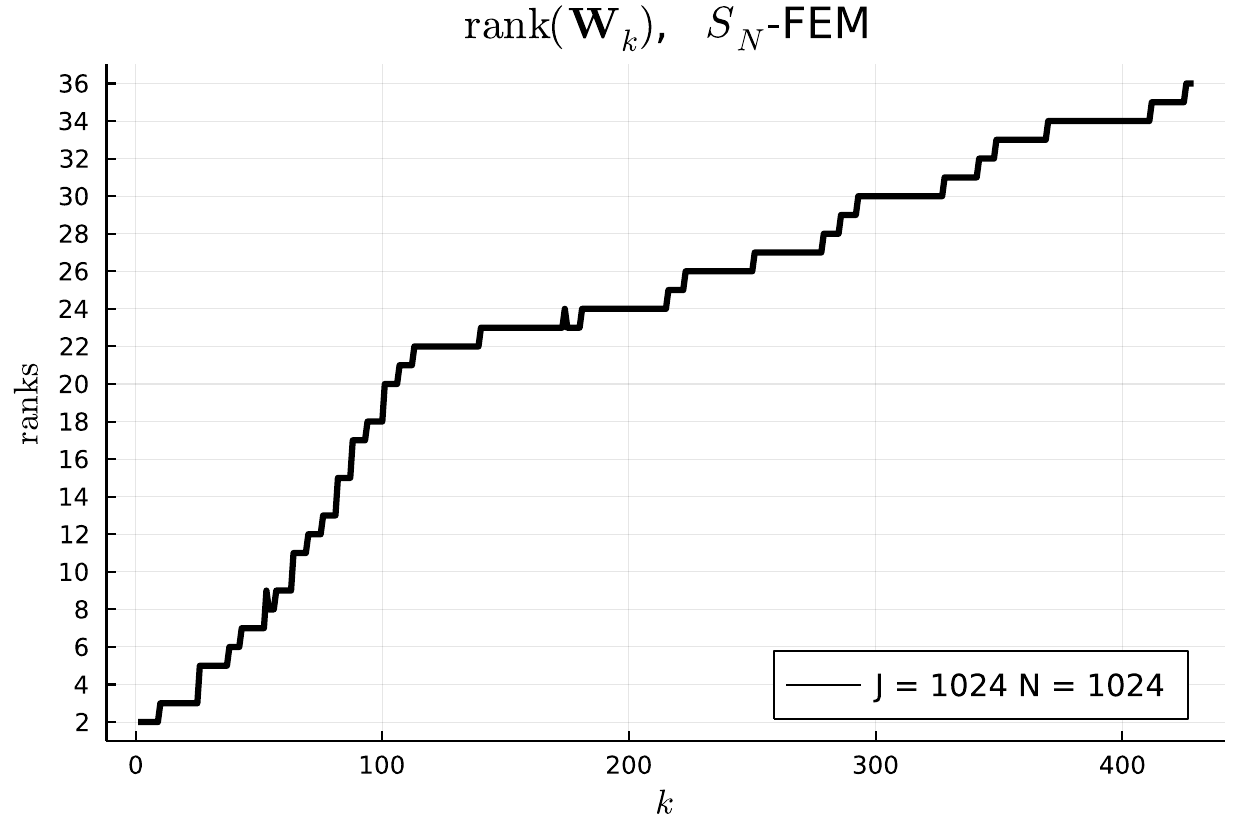}\\
	\includegraphics[width=.49\textwidth]{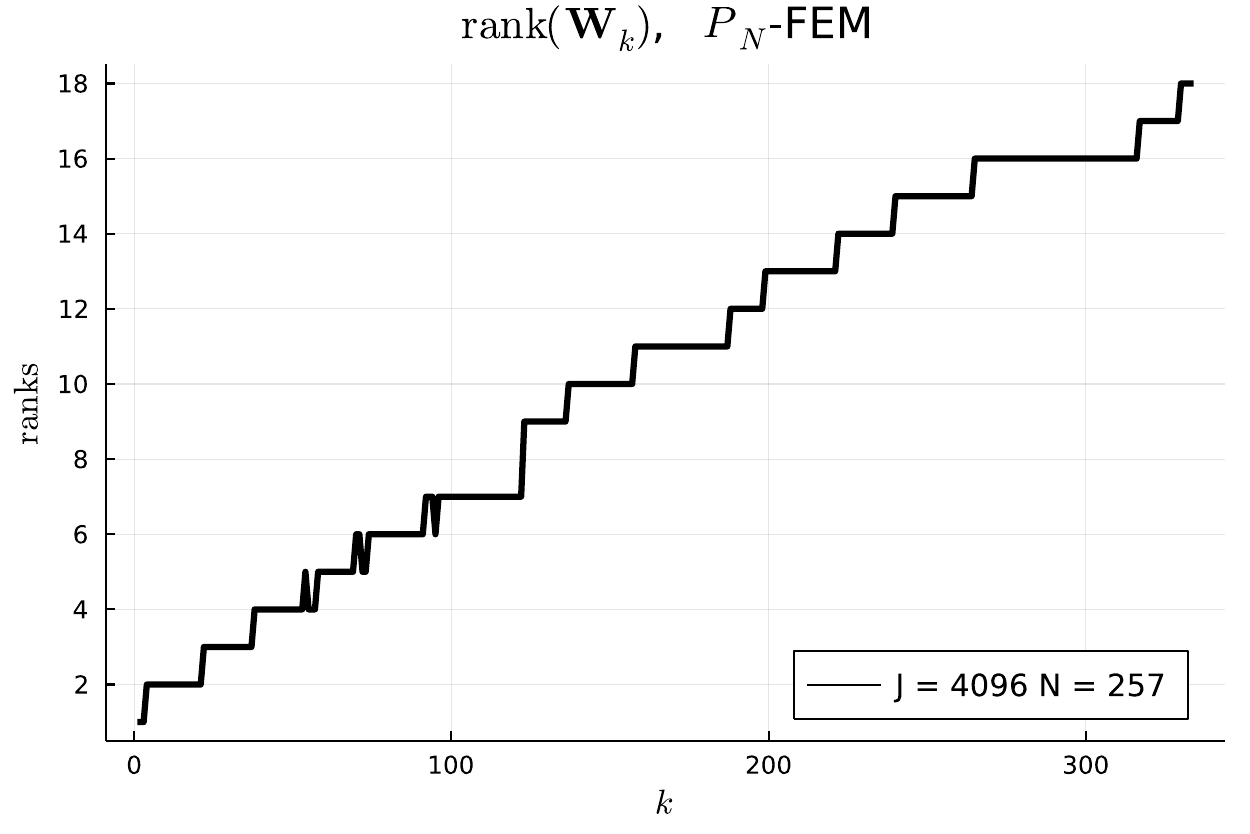}
    \includegraphics[width=.49\textwidth]{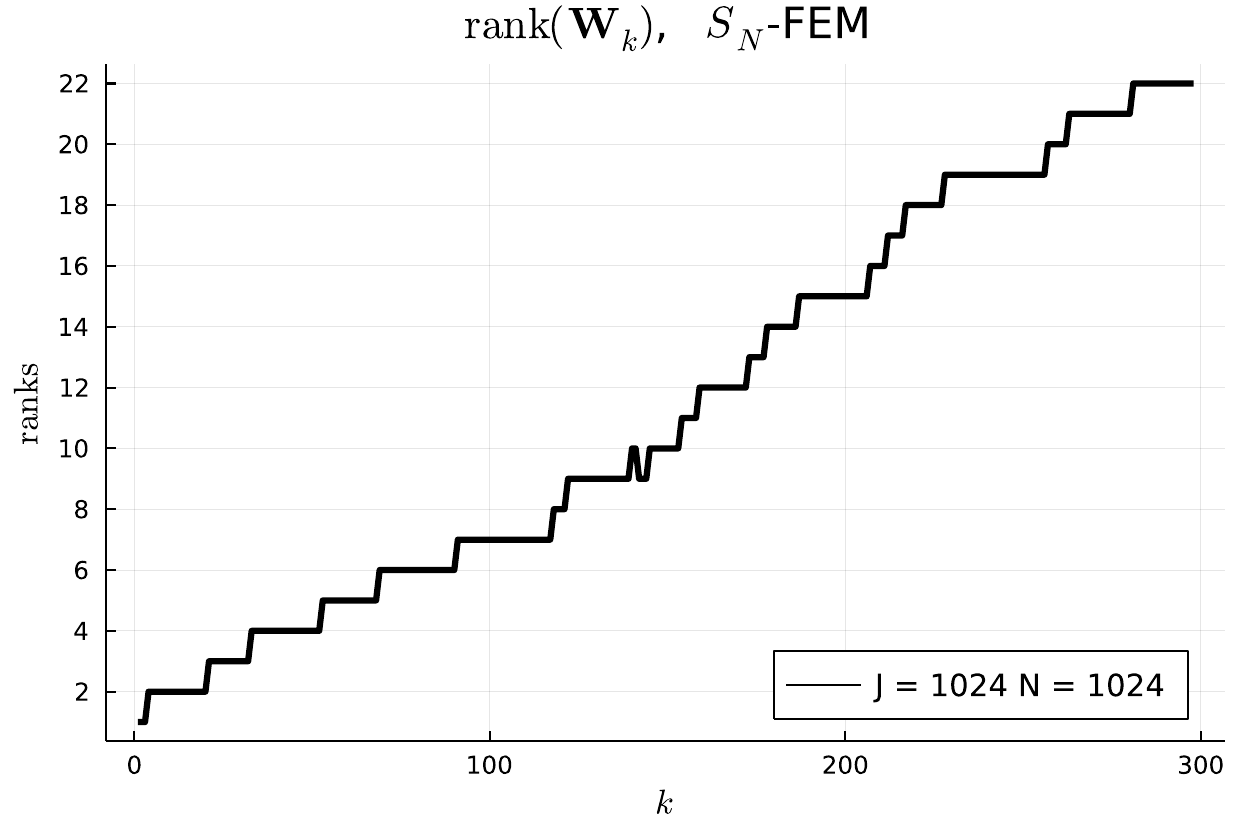}
	\caption{Behaviour of iterates ranks for the $P_N$-FEM and $S_N$-FEM methods applied to test case \cref{eq:test_case_2} with algebraic (top row) and exponential (bottom row) decay of the singular values.}
\end{figure}

\subsection{A physical problem}
\label{sec:A_physical_problem}
In this last example we test our method on a test case inspired by an imaging application. Consider an infrared beam (wavelength $\lambda_{\mathrm{ir}} = \SI{800}{\nano\meter}$) hitting the human body perpendicular to the skin surface. We assume a Gaussian-like inflow specific intensity on the left part of the slab, vanishing on the right part, i.e.,
\begin{equation*}
    g(0, \mu) = \alpha e^{-(1-\mu)^2/\beta}, \qquad g(1, \mu) = 0.
\end{equation*}
We assume there are no internal sources of radiation, $q \equiv 0$. Let us consider a slab of total length $\SI{20}{\milli\meter}$ constituted by three layers: skin, blood, and muscle. We choose the scattering and absorption parameters to be piecewise constant functions
\begin{equation}
    \label{eq:test_case_3}
    \begin{aligned}
        \sigma_s(z) & = 36.52\chi_{[0, 0.75)}(z) + 32.27\chi_{[0.75, 0.875)}(z) + 5.20\chi_{[0.875, 1]}(z) \\
        \sigma_a(z) & = 0.52\chi_{[0, 0.75)}(z) + 8.31\chi_{[0.75, 0.875)}(z) + 0.60\chi_{[0.875, 1]}(z),
    \end{aligned}
\end{equation}
which reflects realistic width for the layers ($\sim \SI{15}{\milli\meter}$ for skin, $\sim \SI{2.5}{\milli\meter}$ for blood and $\sim \SI{2.5}{\milli\meter}$ for muscle), cf., e.g., \cite{ARC05}. Given this construction and following \cite{ZL07}, realistic values for $\alpha$ and $\beta$ are $\alpha = 2.4$ and $\beta = 2500$.
These optical parameters lead to $\rho=0.9996$, when using $\epsilon=1/10$ to set up the preconditioner.

For this test case no analytic solution is available. However, due to the strong concentration of the source term $g$ around $\mu=1$, we expect a high value of $N$ is required in the $P_N$-FEM method to obtain accurate solutions.
We run our algorithm with a Richardson tolerance set to $\e = 10^{-4}$, expecting an iteration count of $\log(\varepsilon)/\log(\rho)\approx 23\,000$. As shown in \cref{tab:TC3_PN_SN_FEM} and in the top-left plots of \cref{fig:TC3_plots_PN} and \cref{fig:TC3_plots_SN}, we confirm the ability of our method to keep the ranks of the iterates small during the procedure. Decay of the residuals (bottom-left plots) allows to control the convergence of the method. The bottom-right plots show the decay of the singular values of the output $\WW^\e$ , as well as the one of its back-transformed $u_{J,N}$ through the preconditioner.  

\begin{table}[ht!]
    \label{tab:TC3_PN_SN_FEM}
    \caption{Number of iterations ($N_{\mathrm{it}}$) of \Cref{alg:dynamical_thresholding} for the $P_N$-FEM (left) and $S_N$-FEM (right) methods (with $J$ and $N$ discretization parameters according to \Cref{sec:PN-FEM_method} and \Cref{sec:SN-FEM_method}, respectively) applied to test case \cref{eq:test_case_3}, and final rank of the output of the algorithm and of its back transformation through the preconditioner.}
    \centering\small\setlength\tabcolsep{0.75em}
    \begin{tabular}{c c c c c} 
	   \toprule
	   \multicolumn{5}{c}{$P_N$-FEM method}\\
	   \cmidrule{1-5}
        $J$ & $N$ & $N_{\mathrm{it}}$ & $\mathtt{r}(\WW^{\e})$ & $\mathtt{r}(\UU^\e)$ \\
        \midrule
        128  & 27  & 15154 & 12 & 14 \\
        256  & 41  & 15553 & 13 & 21 \\
        512  & 65  & 15591 & 14 & 26 \\
        1024 & 103 & 15698 & 15 & 29 \\
        2048 & 163 & 15805 & 17 & 31 \\ 
        4096 & 257 & 15792 & 18 & 35 \\ 
        \bottomrule
    \end{tabular}
    \quad
    \begin{tabular}{c c c c c} 
	   \toprule
	   \multicolumn{5}{c}{$S_N$-FEM method}\\
	   \cmidrule{1-5}
        $J$ & $N$ & $N_{\mathrm{it}}$ & $\mathtt{r}(\WW^{\e})$ & $\mathtt{r}(\UU^\e)$ \\
        \midrule
        128  & 128  & 13857 & 15 & 33 \\
        256  & 256  & 14041 & 17 & 37 \\
        512  & 512  & 14147 & 18 & 40 \\
        1024 & 1024 & 14236 & 19 & 44 \\
        %
        \bottomrule
    \end{tabular}
\end{table}

\setlength{\textfloatsep}{5pt plus 1.0pt minus 2.0pt}

\begin{figure}[ht!]
    \centering
    \label{fig:TC3_plots_PN}
	\includegraphics[width=.45\textwidth]{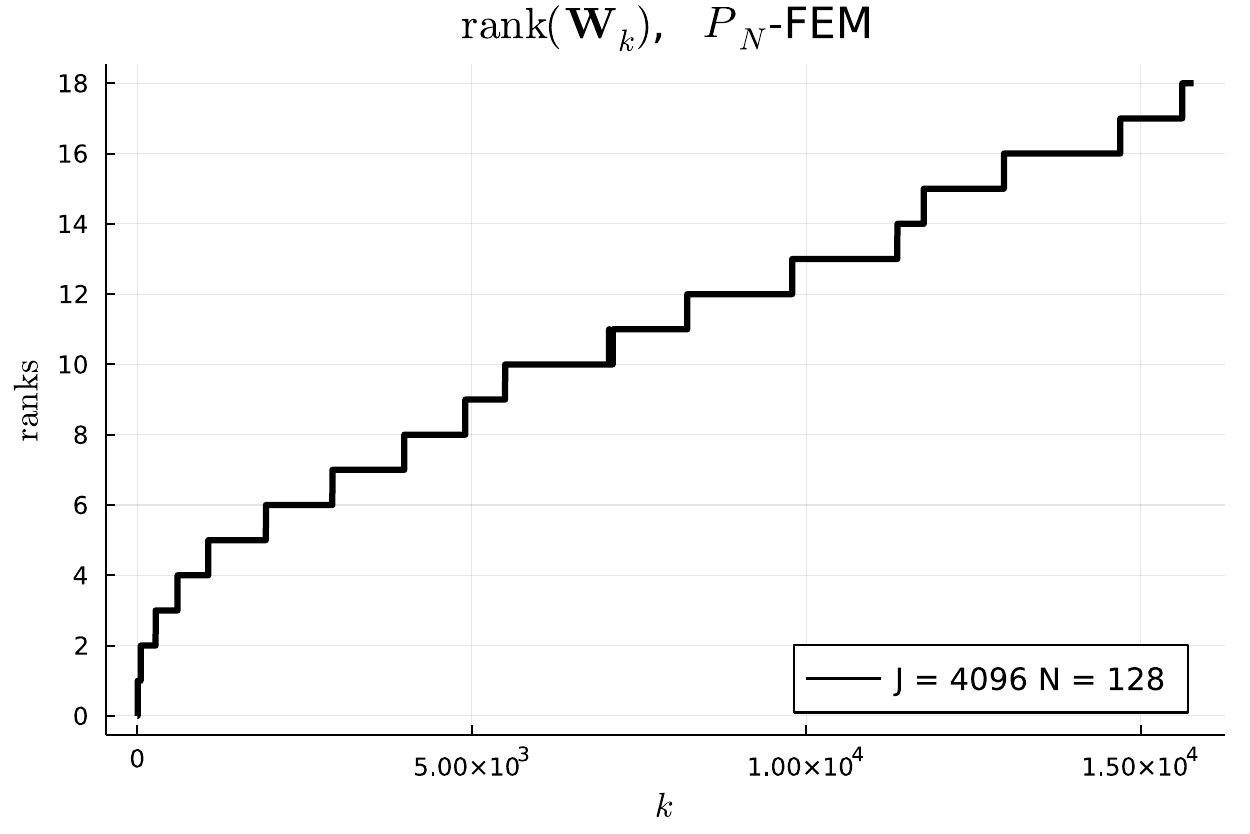}
	\includegraphics[width=.45\textwidth]{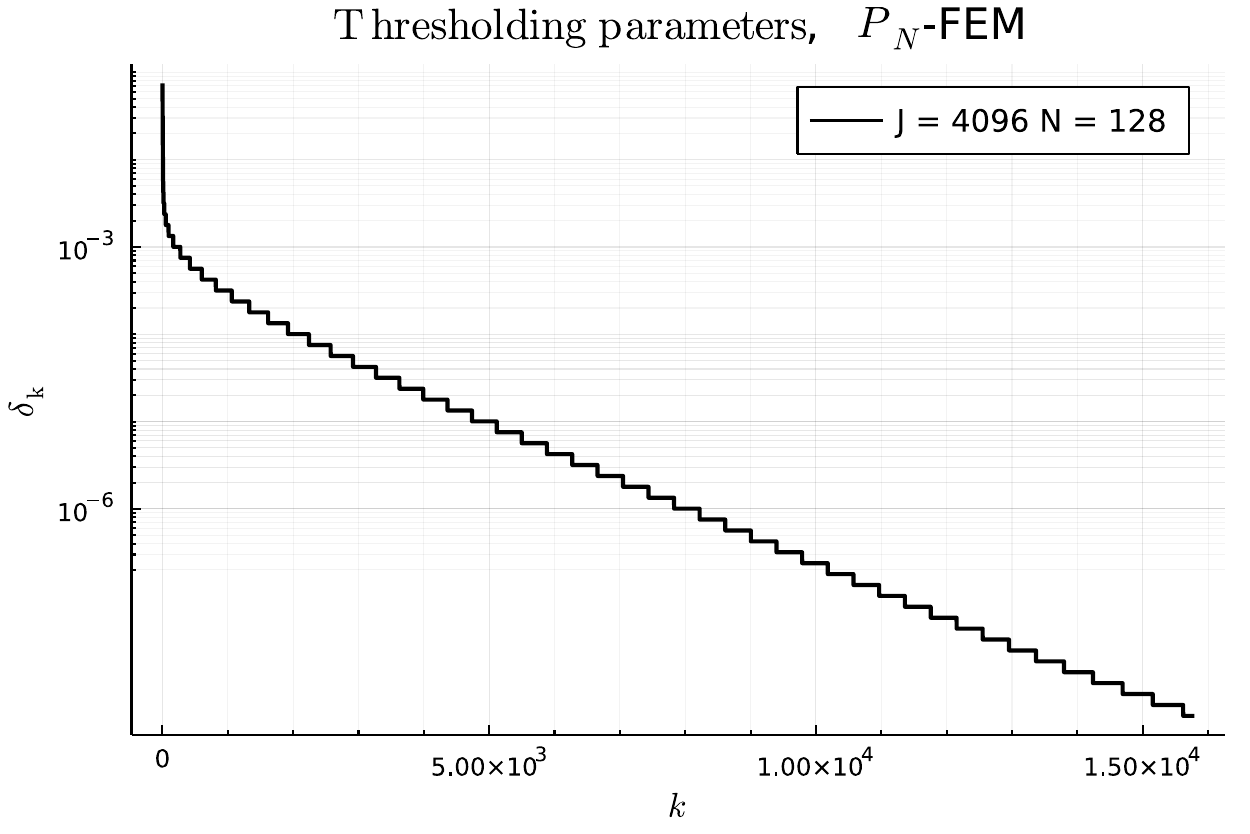}\\
	\includegraphics[width=.45\textwidth]{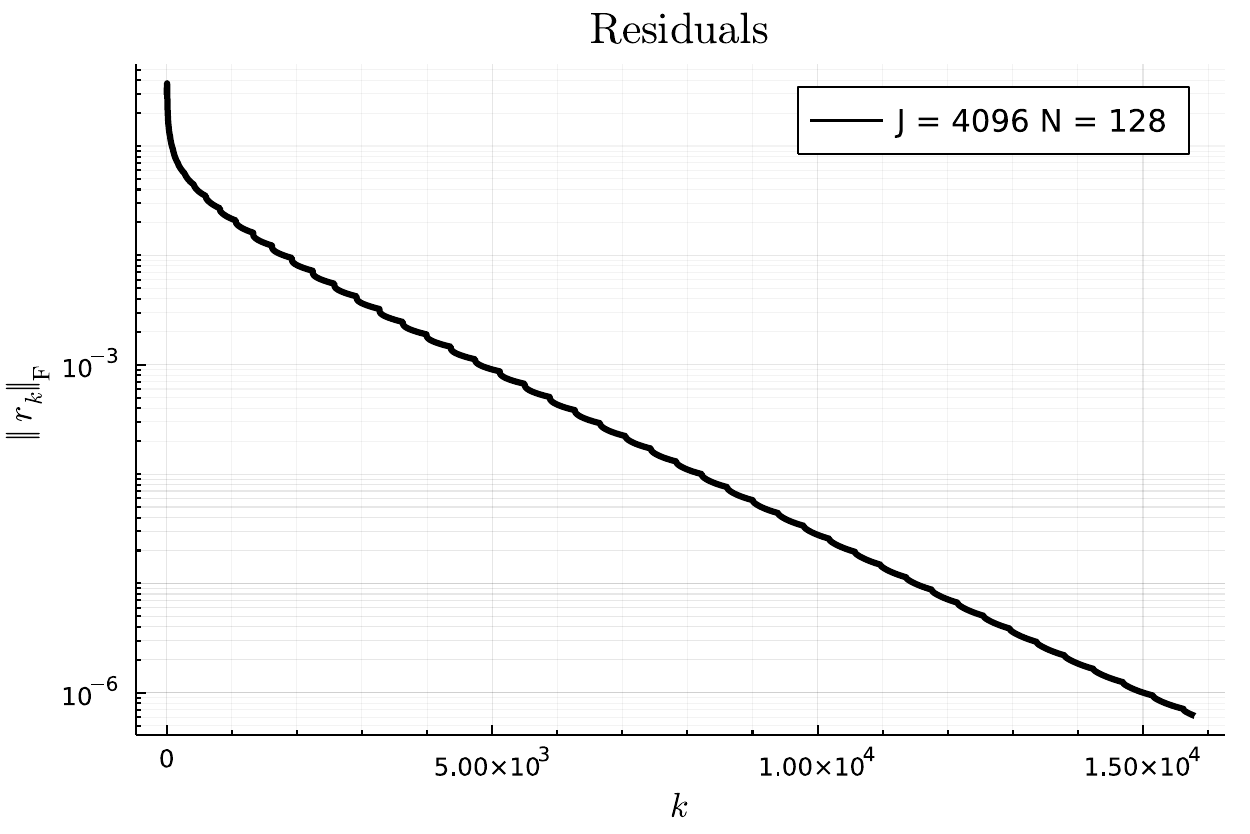}
    \includegraphics[width=.45\textwidth]{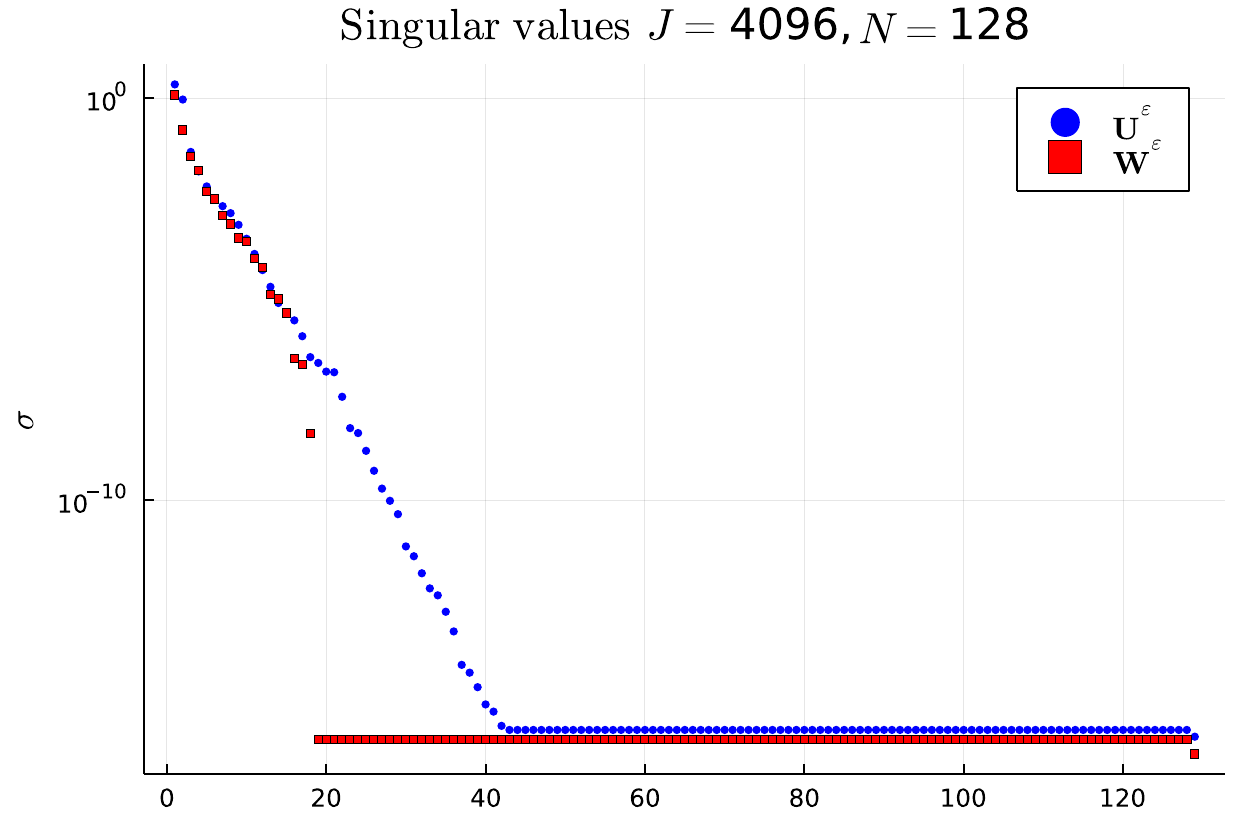}
	\caption{Behavior of iterates ranks (top left), thresholding parameter (top right), Frobenius norm of the residuals (bottom left) and singular values of both $\WW^{\e}$ and $\UU^\e$ (bottom right), for the $P_N$-FEM method applied to test case \cref{eq:test_case_3}.}
\end{figure}

\begin{figure}[ht!]
    \centering
    \label{fig:TC3_plots_SN}
	\includegraphics[width=.45\textwidth]{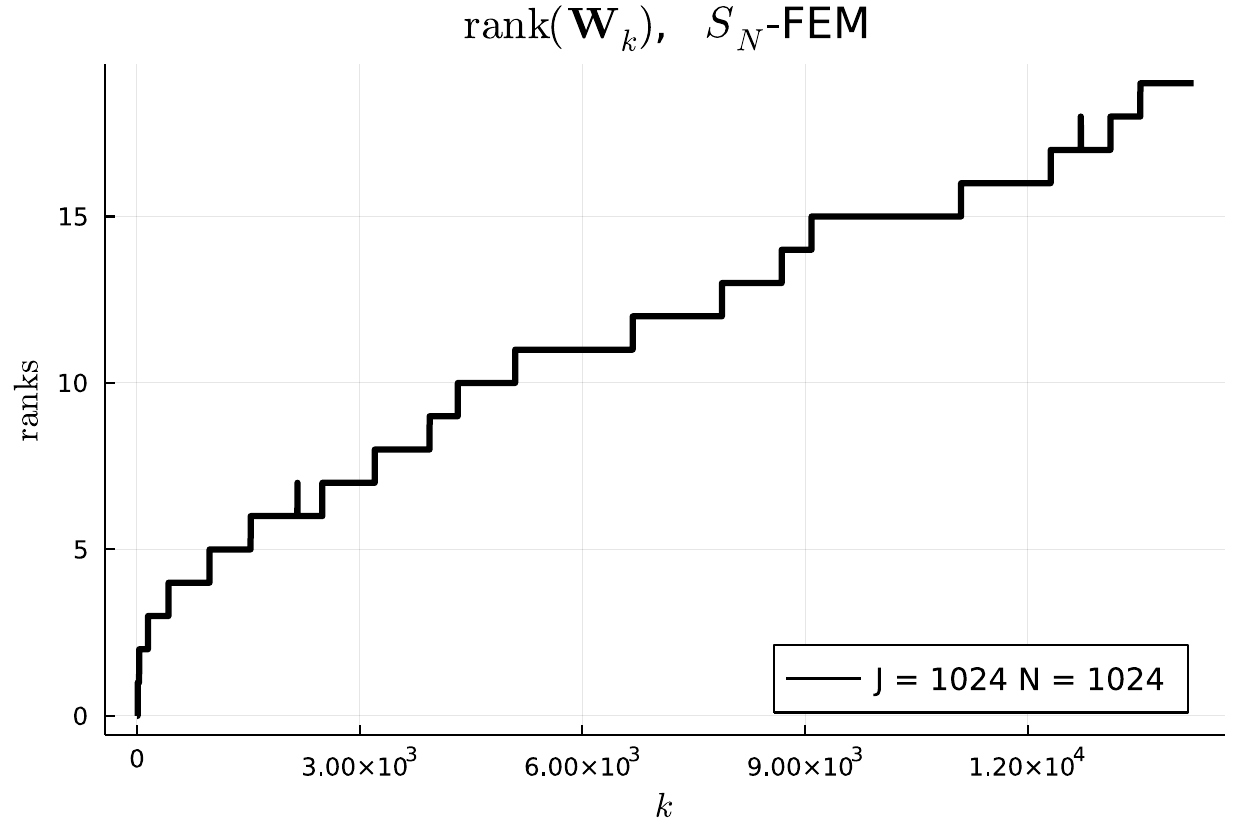}
	\includegraphics[width=.45\textwidth]{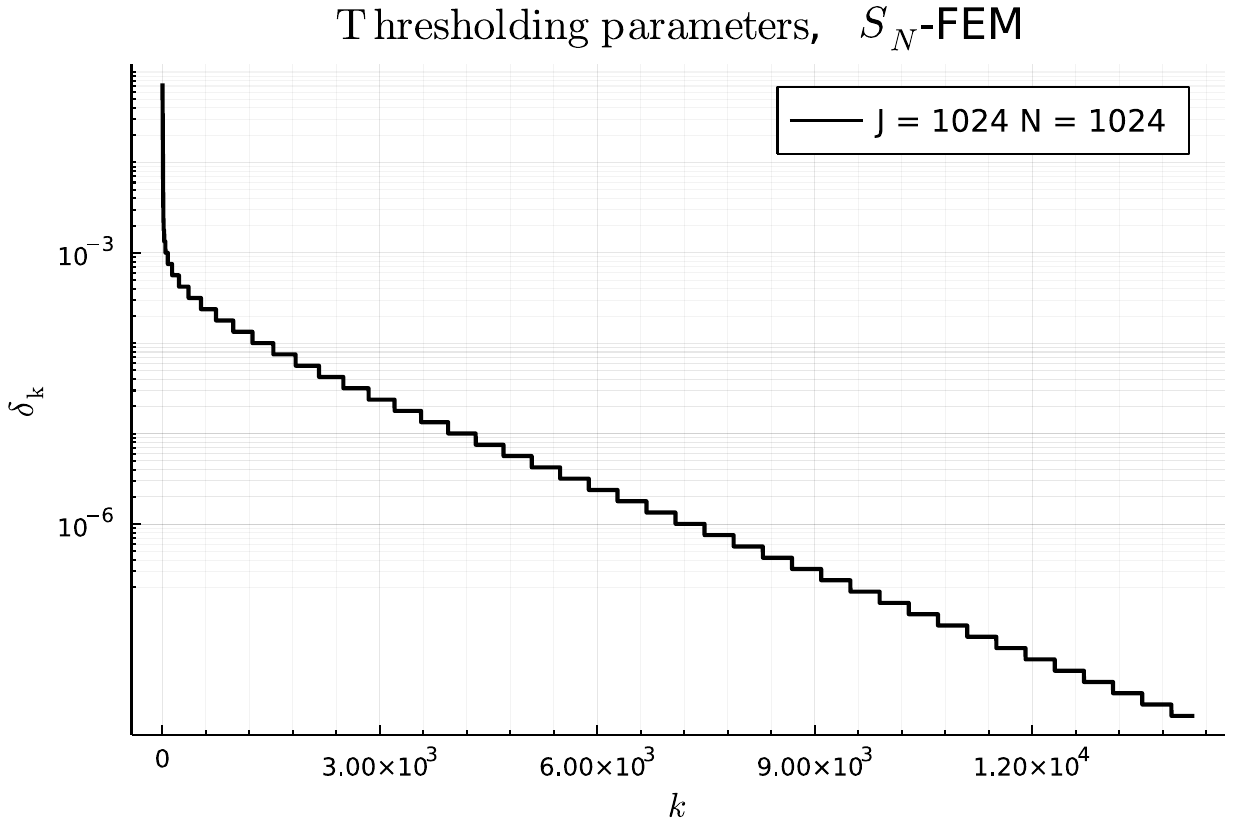}\\
	\includegraphics[width=.45\textwidth]{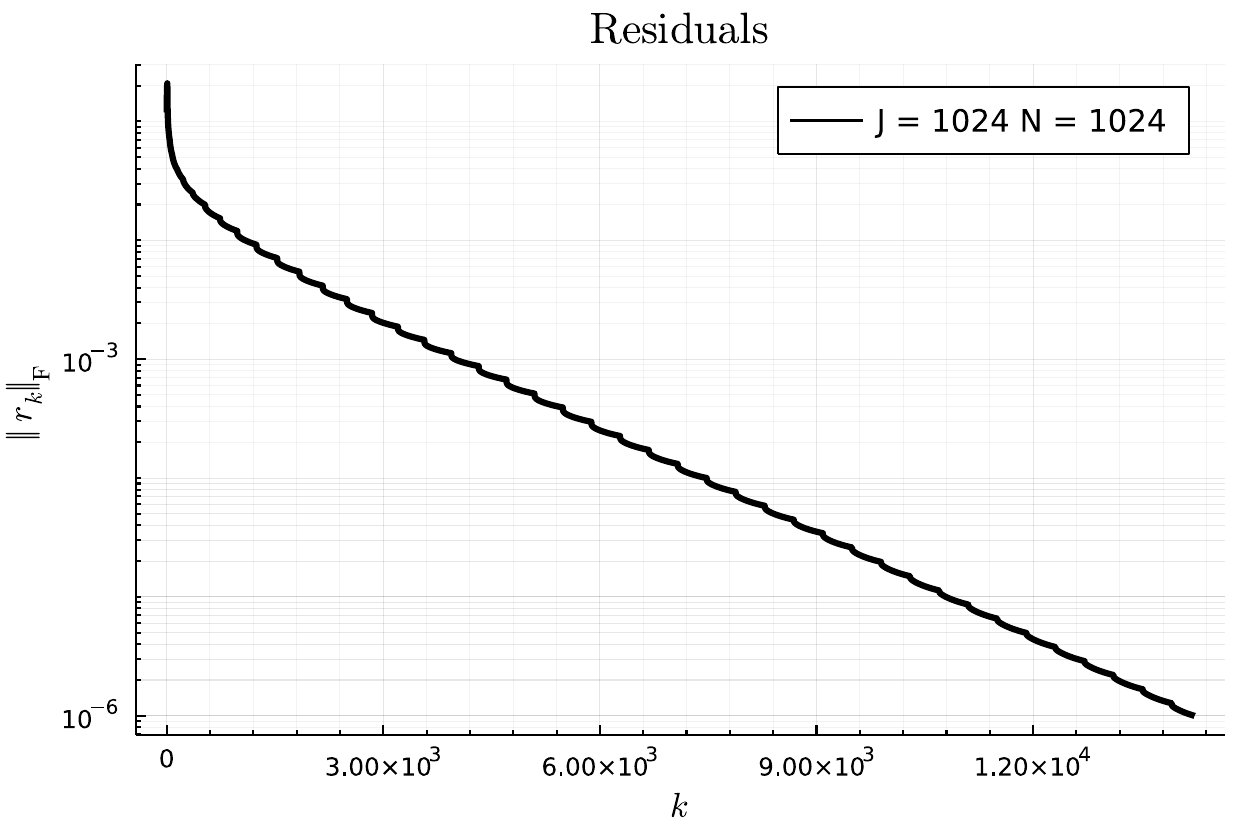}
    \includegraphics[width=.45\textwidth]{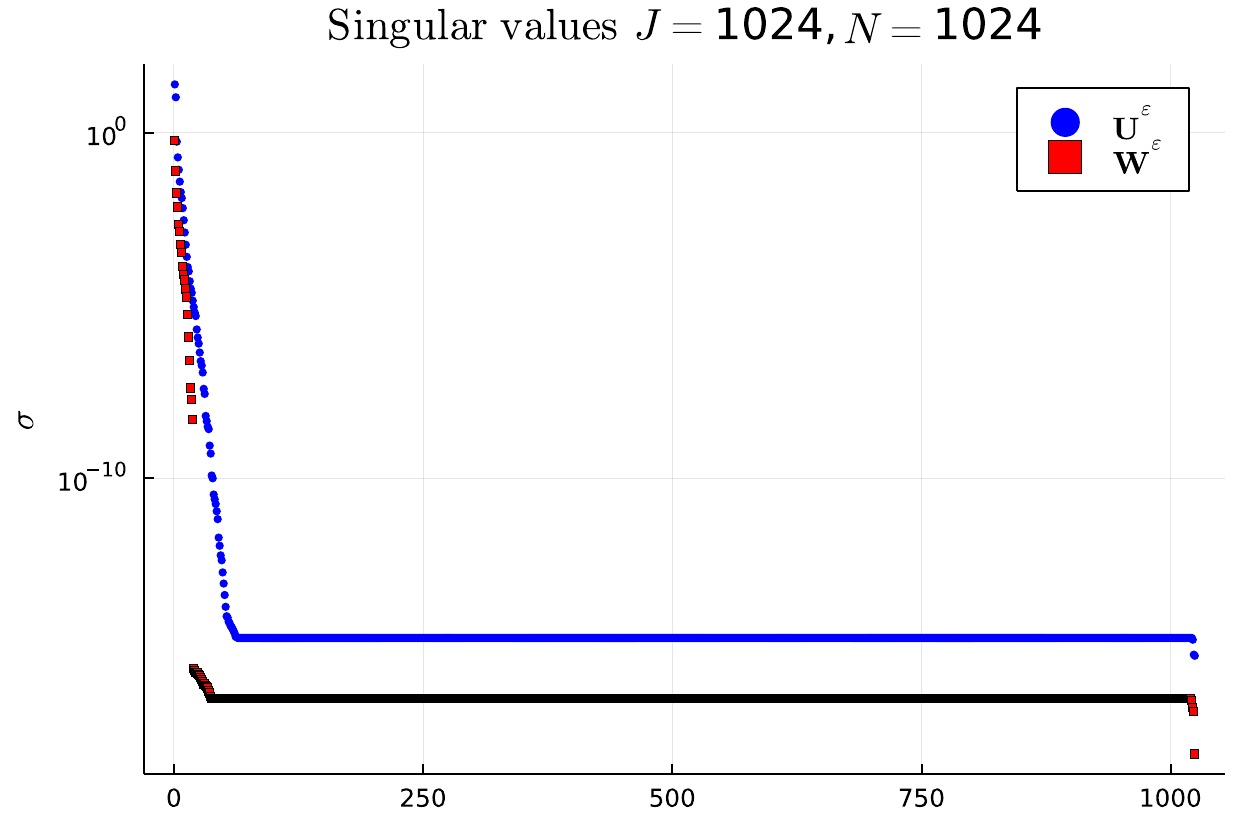}
	\caption{Behavior of iterates ranks (top left), thresholding parameter (top right), Frobenius norm of the residuals (bottom left) and singular values of both $\WW^{\e}$ and $\UU^\e$ (bottom right), for the $S_N$-FEM method applied to test case \cref{eq:test_case_3}.}
\end{figure}

\section{Conclusions and possible extensions}
\label{sec:conclusion}
We have developed a rigorous low-rank framework for the even-parity formulation of the stationary monochromatic radiative transfer equation in slab geometry, which can be applied to many discretization schemes as exemplified by considering widely used $P_N$ and $S_N$ schemes. Our framework employs iterative computations that efficiently maintain low ranks for the iterates.
To keep the total number of iterations low, we have devised a preconditioner in Kronecker-sum format that is based on exponential sum approximations. 
Moreover, the preconditioner allows control over errors in energy norm by performing computations in corresponding Euclidean spaces. For the exemplary considered discretizations, we made all hyperparameters, such as the choice of the ranks for the preconditioner, explicit.
We showed that inexact but accurate application of the preconditioned linear operator allows us to further reduce intermediate ranks, at the expense of somewhat increased iteration counts. 
To overcome the slow convergence for diffusion dominated problems like the test case studied in \Cref{sec:A_physical_problem}, one could investigate low-rank updates computed in a manner that is similar to more established acceleration schemes, e.g., as in \cite{AL02,PS20,DPS22}.

\new{It is tempting to accelerate the iterative scheme by replacing the investigated preconditioned Richardson iteration by other preconditioned Krylov space methods in low-rank format \cite{KT11,BG13,D13}. However, orthogonality properties, which are essential for fast Krylov space methods such as the conjugate gradient method, are usually not preserved under rank truncation; hence, such methods are not guaranteed to perform better in terms of iteration count \cite{AU20}. Moreover, it is more difficult to control the ranks of iterates than for simpler iterative schemes, and avoiding sharp rank growth especially in the initial steps of Krylov space methods can be challenging; see, e.g., \cite[Ch.~4]{T12} and \cite{KPT14}. We leave a more detailed investigation of these trade-offs for future work.}

\section*{Acknowledgements}
R.B. and M.S. acknowledge support by the Dutch Research Council (NWO) via grant OCENW.KLEIN.183.

\newpage
\bibliographystyle{siamplain}
\bibliography{references_FINAL}

\begin{thebibliography}{10}

\bibitem{AL02}
{\sc M.~L. Adams and E.~W. Larsen}, {\em Fast iterative methods for
  discrete-ordinates particle transport calculations}, Progress in Nuclear
  Energy, 40(1) (2002), pp.~3--159,
  \url{https://doi.org/https://doi.org/10.1016/S0149-1970(01)00023-3}.

\bibitem{A98}
{\sc V.~Agoshkov}, {\em Boundary value problems for transport equations},
  Model. Simul. Sci. Eng. Technol., Birkh\"{a}user Boston, Inc., Boston, MA,
  1998, \url{https://doi.org/10.1007/978-1-4612-1994-1}.

\bibitem{AGH16}
{\sc B.~Ahmedov, M.~A. Grepl, and M.~Herty}, {\em Certified reduced-order
  methods for optimal treatment planning}, Math. Models Methods Appl. Sci., 26
  (2016), pp.~699--727, \url{https://doi.org/10.1142/S0218202516500159}.

\bibitem{ARC05}
{\sc G.~Alexandrakis, F.~R. Rannou, and C.~F. Arion}, {\em Tomographic
  bioluminescence imaging by use of a combined optical-pet (opet) system: a
  computer simulation feasibility study}, Phys. Med. Biol., 51 (2005),
  pp.~391--409, \url{https://doi.org/10.1080/10407790600964583}.

\bibitem{AU20}
{\sc M.~Ali and K.~Urban}, {\em H{T}-{AWGM}: a hierarchical {T}ucker-adaptive
  wavelet {G}alerkin method for high-dimensional elliptic problems}, Adv.
  Comput. Math., 46 (2020), pp.~Paper No. 59, 34,
  \url{https://doi.org/10.1007/s10444-020-09797-9}.

\bibitem{AS09}
{\sc S.~R. Arridge and J.~C. Schotland}, {\em Optical tomography: forward and
  inverse problems}, Inverse Problems, 25 (2009), pp.~123010, 59,
  \url{https://doi.org/10.1088/0266-5611/25/12/123010}.

\bibitem{B23}
{\sc M.~Bachmayr}, {\em Low-rank tensor methods for partial differential
  equations}, Acta Numer., 32 (2023), pp.~1--121,
  \url{https://doi.org/10.1017/S0962492922000125}.

\bibitem{BD16}
{\sc M.~Bachmayr and W.~Dahmen}, {\em Adaptive low-rank methods: problems on
  {S}obolev spaces}, SIAM J. Numer. Anal., 54 (2016), pp.~744--796,
  \url{https://doi.org/10.1137/140978223}.

\bibitem{BS17}
{\sc M.~Bachmayr and R.~Schneider}, {\em Iterative methods based on soft
  thresholding of hierarchical tensors}, Found. Comput. Math., 17 (2017),
  pp.~1037--1083, \url{https://doi.org/10.1007/s10208-016-9314-z}.

\bibitem{BG13}
{\sc J.~Ballani and L.~Grasedyck}, {\em A projection method to solve linear
  systems in tensor format}, Numer. Linear Algebra Appl., 20 (2013),
  pp.~27--43, \url{https://doi.org/10.1002/nla.1818},
  \url{https://doi.org/10.1002/nla.1818}.

\bibitem{BBPS23}
{\sc R.~Bardin, F.~Bertrand, O.~Palii, and M.~Schlottbom}, {\em A phase-space
  discontinuous galerkin approximation for the radiative transfer equation in
  slab geometry}, Computational Methods in Applied Mathematics,  (2024),
  \url{https://doi.org/10.1515/cmam-2023-0090}.

\bibitem{BM10}
{\sc G.~Beylkin and L.~Monz\'{o}n}, {\em Approximation by exponential sums
  revisited}, Appl. Comput. Harmon. Anal., 28 (2010), pp.~131--149,
  \url{https://doi.org/10.1016/j.acha.2009.08.011}.

\bibitem{CZ67}
{\sc K.~M. Case and P.~F. Zweifel}, {\em Linear transport theory},
  Addison-Wesley Publishing Co., Reading, Mass.-London-Don Mills, Ont., 1967.

\bibitem{C60}
{\sc S.~Chandrasekhar}, {\em Radiative transfer}, Dover Publications, Inc., New
  York, 1960.

\bibitem{DGM18}
{\sc W.~Dahmen, F.~Gruber, and O.~Mula}, {\em An adaptive nested source term
  iteration for radiative transfer equations}, Math. Comp., 89 (2020),
  pp.~1605--1646, \url{https://doi.org/10.1090/mcom/3505}.

\bibitem{DV98}
{\sc R.~A. DeVore}, {\em Nonlinear approximation}, Acta Numer., 7 (1998),
  pp.~51--150, \url{https://doi.org/10.1017/S0962492900002816}.

\bibitem{D13}
{\sc S.~V. Dolgov}, {\em T{T}-{GMRES}: solution to a linear system in the
  structured tensor format}, Russian J. Numer. Anal. Math. Modelling, 28
  (2013), pp.~149--172, \url{https://doi.org/10.1515/rnam-2013-0009}.

\bibitem{DES21}
{\sc J.~D{\"o}lz, H.~Egger, and M.~Schlottbom}, {\em A model reduction approach
  for inverse problems with operator valued data}, Numerische Mathematik, 148
  (2021), pp.~889--917, \url{https://doi.org/10.1007/s00211-021-01224-5},
  \url{http://dx.doi.org/10.1007/s00211-021-01224-5}.

\bibitem{DPS22}
{\sc J.~D\"olz, O.~Palii, and M.~Schlottbom}, {\em On robustly convergent and
  efficient iterative methods for anisotropic radiative transfer}, J. Sci.
  Comput., 90 (2022), pp.~Paper No. 94, 28,
  \url{https://doi.org/10.1007/s10915-021-01757-9},
  \url{https://doi.org/10.1007/s10915-021-01757-9}.

\bibitem{DM79}
{\sc J.~J. Duderstadt and W.~R. Martin}, {\em Transport theory}, A
  Wiley-Interscience Publication, John Wiley \& Sons, New
  York-Chichester-Brisbane, 1979.

\bibitem{ES12}
{\sc H.~Egger and M.~Schlottbom}, {\em A mixed variational framework for the
  radiative transfer equation}, Math. Models Methods Appl. Sci., 22 (2012),
  pp.~1150014, 30, \url{https://doi.org/10.1142/S021820251150014X}.

\bibitem{ES19}
{\sc H.~Egger and M.~Schlottbom}, {\em A perfectly matched layer approach for
  {$P_N$}-approximations in radiative transfer}, SIAM J. Numer. Anal., 57
  (2019), pp.~2166--2188, \url{https://doi.org/10.1137/18M1172521},
  \url{https://doi.org/10.1137/18M1172521}.

\bibitem{EKKMQ24}
{\sc L.~Einkemmer, K.~Kormann, J.~Kusch, R.~G. McClarren, and J.-M. Qiu}, {\em
  A review of low-rank methods for time-dependent kinetic simulations},
  arxi:2412.0591,  (2024).

\bibitem{Khaza2024}
{\sc H.~{El Kahza}, W.~Taitano, J.-M. Qiu, and L.~Chac{\'o}n}, {\em
  Krylov-based adaptive-rank implicit time integrators for stiff problems with
  application to nonlinear fokker-planck kinetic models}, Journal of
  Computational Physics, 518 (2024), p.~113332,
  \url{https://doi.org/https://doi.org/10.1016/j.jcp.2024.113332},
  \url{https://www.sciencedirect.com/science/article/pii/S0021999124005801}.

\bibitem{E98}
{\sc K.~F. Evans}, {\em The spherical harmonics discrete ordinate method for
  three-dimensional atmospheric radiative transfer}, J. Atmospheric Sci., 55
  (1998), pp.~429--446,
  \url{https://doi.org/10.1175/1520-0469(1998)055<0429:TSHDOM>2.0.CO;2}.

\bibitem{GHS07}
{\sc T.~Gantumur, H.~Harbrecht, and R.~Stevenson}, {\em An optimal adaptive
  wavelet method without coarsening of the iterands}, Math. Comp., 76 (2007),
  pp.~615--629, \url{https://doi.org/10.1090/S0025-5718-06-01917-X}.

\bibitem{GG08}
{\sc W.~Gautschi and C.~Giordano}, {\em Luigi {G}atteschi's work on asymptotics
  of special functions and their zeros}, Numer. Algorithms, 49 (2008),
  pp.~11--31, \url{https://doi.org/10.1007/s11075-008-9208-5}.

\bibitem{GS11b}
{\sc K.~Grella and C.~Schwab}, {\em Sparse discrete ordinates method in
  radiative transfer}, Comput. Methods Appl. Math., 11 (2011), pp.~305--326,
  \url{https://doi.org/10.2478/cmam-2011-0017}.

\bibitem{GS11a}
{\sc K.~Grella and C.~Schwab}, {\em Sparse tensor spherical harmonics
  approximation in radiative transfer}, J. Comput. Phys., 230 (2011),
  pp.~8452--8473, \url{https://doi.org/10.1016/j.jcp.2011.07.028}.

\bibitem{GuoEmaQiu2024}
{\sc W.~Guo, J.~F. Ema, and J.-M. Qiu}, {\em A local macroscopic conservative
  (lomac) low rank tensor method with the discontinuous galerkin method for the
  vlasov dynamics}, Communications on Applied Mathematics and Computation, 6
  (2024), pp.~550--575, \url{https://doi.org/10.1007/s42967-023-00277-7}.

\bibitem{GuoQiu2022}
{\sc W.~Guo and J.-M. Qiu}, {\em A low rank tensor representation of linear
  transport and nonlinear vlasov solutions and their associated flow maps},
  Journal of Computational Physics, 458 (2022), p.~111089,
  \url{https://doi.org/https://doi.org/10.1016/j.jcp.2022.111089}.

\bibitem{GuoQiu2024}
{\sc W.~Guo and J.-M. Qiu}, {\em A conservative low rank tensor method for the
  vlasov dynamics}, SIAM Journal on Scientific Computing, 46 (2024),
  pp.~A232--A263, \url{https://doi.org/10.1137/22m1473960}.

\bibitem{KochLubich2007}
{\sc O.~Koch and C.~Lubich}, {\em Dynamical low‐rank approximation}, SIAM
  Journal on Matrix Analysis and Applications, 29 (2007), pp.~434--454,
  \url{https://doi.org/10.1137/050639703}.

\bibitem{KL15}
{\sc J.~K\'{o}ph\'{a}zi and D.~Lathouwers}, {\em A space-angle {DGFEM} approach
  for the {B}oltzmann radiation transport equation with local angular
  refinement}, J. Comput. Phys., 297 (2015), pp.~637--668,
  \url{https://doi.org/10.1016/j.jcp.2015.05.031}.

\bibitem{KPT14}
{\sc D.~Kressner, M.~Plešinger, and C.~Tobler}, {\em A preconditioned low-rank
  {CG} method for parameter-dependent {L}yapunov matrix equations}, Numer.
  Linear Algebra Appl., 21 (2014), pp.~666--684,
  \url{https://doi.org/10.1002/nla.1919},
  \url{https://doi.org/10.1002/nla.1919}.

\bibitem{KT11}
{\sc D.~Kressner and C.~Tobler}, {\em Low-rank tensor {K}rylov subspace methods
  for parametrized linear systems}, SIAM J. Matrix Anal. Appl., 32 (2011),
  pp.~1288--1316, \url{https://doi.org/10.1137/100799010},
  \url{https://doi.org/10.1137/100799010}.

\bibitem{KS23}
{\sc J.~Kusch and P.~Stammer}, {\em A robust collision source method for rank
  adaptive dynamical low-rank approximation in radiation therapy}, ESAIM Math.
  Model. Numer. Anal., 57 (2023), pp.~865--891,
  \url{https://doi.org/10.1051/m2an/2022090}.

\bibitem{LT03}
{\sc S.~Larrson and V.~Thom{\'e}e}, {\em Partial Differential Equations with
  Numerical Methods}, Texts in Applied Mathematics, Springer Berlin, 2003.

\bibitem{L97}
{\sc E.~W. Larsen}, {\em The nature of transport calculations used in radiation
  oncology}, Transport Theor. Stat., 26 (1997), pp.~739--763,
  \url{https://doi.org/10.1080/00411459708224421}.

\bibitem{MWTLS17}
{\sc X.~Meng, S.~Wang, G.~Tang, J.~Li, and C.~Sun}, {\em Stochastic parameter
  estimation of heterogeneity from crosswell seismic data based on the mont
  carlo radiative transfer theory}, J. Geophys. Eng., 14 (2017), pp.~621--633,
  \url{https://doi.org/10.1088/1742-2140/aa6130}.

\bibitem{PS20}
{\sc O.~Palii and M.~Schlottbom}, {\em On a convergent {DSA} preconditioned
  source iteration for a {DGFEM} method for radiative transfer}, Comput. Math.
  Appl., 79 (2020), pp.~3366--3377,
  \url{https://doi.org/10.1016/j.camwa.2020.02.002}.

\bibitem{PMF20}
{\sc Z.~Peng, R.~G. McClarren, and M.~Frank}, {\em A low-rank method for
  two-dimensional time-dependent radiation transport calculations}, J. Comput.
  Phys., 421 (2020), pp.~109735, 18,
  \url{https://doi.org/10.1016/j.jcp.2020.109735}.

\bibitem{RPBSN18}
{\sc M.~Rustamzhon, D.~Press, G.~K. Baskaran, S.~Sadeghi, and S.~Nizamoglu},
  {\em Unravelling radiative energy transfer in solid-state lighting}, J. Appl.
  Phys., 123 (2018), \url{https://doi.org/10.1063/1.5008922}.

\bibitem{SY17}
{\sc S.~Scholz and H.~Yserentant}, {\em On the approximation of electronic
  wavefunctions by anisotropic {G}auss and {G}auss-{H}ermite functions}, Numer.
  Math., 136 (2017), pp.~841--874,
  \url{https://doi.org/10.1007/s00211-016-0856-4}.

\bibitem{STS17}
{\sc K.~Stamnes, G.~Thomas, and J.~Stamnes}, {\em Radiative transfer in the
  atmosphere and ocean}, Modeling and Simulation in Science, Engineering and
  Technology, Cambridge University Press, 2nd~ed., 2017,
  \url{https://doi.org/10.1017/9781316148549}.

\bibitem{T12}
{\sc C.~Tobler}, {\em Low-rank Tensor Methods for Linear Systems and Eigenvalue
  Problems}, PhD thesis, ETH Z{\"u}rich, 2012.

\bibitem{T50}
{\sc F.~G. Tricomi}, {\em Sugli zeri dei polinomi sferici ed ultrasferici},
  Ann. Mat. Pura Appl. (4), 31 (1950), pp.~93--97,
  \url{https://doi.org/10.1007/BF02428258}.

\bibitem{WHS08}
{\sc G.~Widmer, R.~Hiptmair, and C.~Schwab}, {\em Sparse adaptive finite
  elements for radiative transfer}, J. Comput. Phys., 227 (2008),
  pp.~6071--6105, \url{https://doi.org/10.1016/j.jcp.2008.02.025}.

\bibitem{Y20}
{\sc H.~Yserentant}, {\em On the expansion of solutions of {L}aplace-like
  equations into traces of separable higher dimensional functions}, Numer.
  Math., 146 (2020), pp.~219--238,
  \url{https://doi.org/10.1007/s00211-020-01138-8}.

\bibitem{ZL07}
{\sc J.~M. Zhao and L.~H. Liu}, {\em Second-order radiative transfer equation
  and its properties of numerical solution using the finite-element method},
  Numer. Heat Transf. B, 51 (2007), pp.~391--409,
  \url{https://doi.org/10.1080/10407790600964583}.

\end{thebibliography}

\end{document}